\documentclass[a4paper,11pt]{amsart}
\usepackage[utf8]{inputenc}
\usepackage[english]{babel}
\usepackage{comment}

\newcommand{\Hepi}{\mathrm{Hepi}}

\DeclareMathOperator{\Lan}{\mathrm{Lan}}
\DeclareMathOperator{\colim}{\mathrm{colim}}
\DeclareMathOperator{\HH}{\mathrm{HH}}
\DeclareMathOperator{\cHH}{\mathcal{HH}}
\DeclareMathOperator{\Hom}{\mathrm{Hom}}

\DeclareMathOperator{\Ext}{\mathrm{Ext}}
\DeclareMathOperator{\Tw}{\mathrm{Tw}}
\DeclareMathOperator{\Tot}{\mathrm{Tot}}
\newcommand{\GS}{\mathrm{GS}}

\usepackage[margin=2.5cm]{geometry}
\usepackage{enumerate}
\usepackage{float}
\usepackage[final]{graphicx}
\usepackage{epstopdf} 

\usepackage{tabularx}
\newcolumntype{C}[1]{>{\centering\arraybackslash}m{#1}}

\usepackage{enumitem}
\usepackage{comment}
\usepackage{soul}

\usepackage{tikz-cd}
\usepackage{amsmath, multirow}
\usepackage{amsfonts}
\usepackage{amssymb}
\usepackage{amsthm}
\usepackage{hyperref}
\usepackage{cleveref}

\usepackage[T1]{fontenc}
\usepackage{mathrsfs}
\usepackage{mathtools}

\usepackage{tikz}
\usetikzlibrary{decorations.markings}
\usetikzlibrary{decorations.pathreplacing}
\usetikzlibrary{arrows,shapes,positioning,patterns}
\usetikzlibrary{knots}
\tikzstyle{none}=[inner sep=0pt]
\pgfdeclarelayer{edgelayer}
\pgfdeclarelayer{nodelayer}
\pgfsetlayers{edgelayer,nodelayer,main}
\usepackage{url} 
\usepackage[all]{xy}

\newcommand{\op}{\mathrm{op}}

\newcommand{\Ab}{\mathbf{Ab}}

\newcommand{\bA}{\mathbf{A}}

\newcommand{\bD}{\mathbf{D}}

\newcommand{\cC}{\mathcal{C}}

\newcommand{\cA}{\mathcal{A}}
\newcommand{\cF}{\mathcal{F}}

\newcommand{\cM}{\mathcal{M}}

\newcommand{\Nerve}{\mathrm{N}}

\newcommand{\bR}{\mathbb R}
\newcommand{\bQ}{\mathbb Q}
\newcommand{\N}{\mathbb N}

\newcommand{\bC}{\mathbf C}

\newcommand{\bZ}{\mathbb Z}

\renewcommand{\H}{\mathrm{H}}

\newcommand{\Mod}{\mathbf{Mod}}
\newcommand{\Alg}{\mathbf{Alg}}

\newcommand{\Tor}{\rm Tor}

\theoremstyle{plain}
\newtheorem{thm}{Theorem}[section]
\newtheorem{prop}[thm]{Proposition}
\newtheorem{lemma}[thm]{Lemma}
\newtheorem{cor}[thm]{Corollary}
\theoremstyle{remark}
\newtheorem{rem}[thm]{Remark}

\newtheorem{example}[thm]{Example} 
\newtheorem{q}[thm]{Question} 
\newtheorem{notation}[thm]{Notation} 
 
\theoremstyle{definition}
\newtheorem{defn}[thm]{Definition}

\newtheorem*{crit*}{Criterion~B}

\definecolor{aquamarine}{rgb}{0.5, 1.0, 0.83}
\definecolor{princetonorange}{rgb}{1.0, 0.56, 0.0}
\definecolor{caribbeangreen}{rgb}{0.0, 0.8, 0.6}
\definecolor{bunired}{rgb}{0.8, 0.0, 0.0}
\definecolor{cdgreen}{rgb}{0.0, 0.42, 0.24}
\definecolor{lavender(floral)}{rgb}{0.71, 0.49, 0.86}
\definecolor{bluedefrance}{rgb}{0.19, 0.55, 0.91}
\definecolor{iris}{rgb}{0.35, 0.31, 0.81}
\definecolor{darkgreen}{rgb}{0.33, 0.42, 0.18}

\newcommand{\tG}{{\tt G}}

\title{Diagrammatic Hochschild cohomology via cohomology of categories, and incidence algebras}
\author{Luigi Caputi}
\author{Francesco Vaccarino }

\begin{document}

\begin{abstract}
In this paper, we study the Hochschild cohomology of diagrams of algebras introduced by Gerstenhaber and Schack and provide computations for filtrations of incidence algebras. Our aims are threefold: firstly, we revisit and explore the connection between the Gerstenhaber-Schack complexes and the Baues-Wirsching cohomology of categories. 
Secondly, we analyse the behaviour of (diagrammatic) Hochschild cohomology in the context of homological epimorphisms of algebras. Thirdly, we 
study diagrams of incidence algebras. In particular, as a main application, we compute the diagrammatic Hochschild cohomology of diagrams of incidence algebras associated to finite filtrations of simplicial complexes. 
\end{abstract}

\maketitle

\section{Introduction}

Hochschild cohomology plays a central role in the study of algebraic structures, capturing both algebraic and homotopical properties. It has been extensively studied in contexts ranging from representation theory to algebraic geometry and topology; see~\cite{loday,witherspoon2019hochschild} for comprehensive introductions. In this work, we focus on diagrammatic Hochschild cohomology, an extension of the classical Hochschild cohomology developed through the framework introduced by Gerstenhaber and Schack~\cite{zbMATH03865538, GERSTENHABER1983143, MR0981619, MR0965749}. The diagrammatic approach, which is also closely related to deformations of algebras, generalizes cohomological computations from algebras to diagrams of algebras (that is, presheaves of algebras indexed by a small category), thereby enabling insights into their interrelationships. As an intriguing example of application, it was shown in~\cite{GERSTENHABER1983143} --  see also~\cite{zbMATH04118524, zbMATH06676099} -- that simplicial cohomology is Hochschild cohomology. Since the seminal works of Gerstenhaber and Schack, diagrammatic Hochschild cohomology has generated increasing interest, and it has been studied and generalized to various contexts; see, among others, eg.~\cite{zbMATH02153391,zbMATH05863900,zbMATH06111448, zbMATH06893235, zbMATH07529817,zbMATH07707528}.
Computations of diagrammatic Hochschild cohomology are generally not easy to handle. To enhance the available computational tools and facilitate applications, this paper utilizes the cohomologies of categories, specifically the Baues-Wirsching cohomology~\cite{BAUES1985187}. We start by revising and highlighting the connection between diagrammatic Hochschild cohomology and Baues-Wirsching cohomology by means of a suitable spectral sequence -- which, to the authors' knowledge, was first due to Robinson~\cite{robinson2008cohomology}. 

 In this paper, for a presheaf $\cA\colon \bC^\op\to \Alg_k$ of $k$-algebras on a finite category $\bC$, we analyse the second page of the aforementioned spectral sequence~\cite{robinson2008cohomology}, which converges to the diagrammatic Hochschild cohomology $\HH_{\GS}^*(\cA,\cA)$ of Gerstenhaber and Schack. It is, in fact, implicit in Gerstenhaber-Schack's definition that diagrammatic Hochschild cohomology can be seen as a  (generalization of a) local cohomology for a twisted system of coefficients over
the nerve of the indexing category -- see also~\cite{robinson2008cohomology} and \cite{zbMATH05364854} for related considerations. This fact can be formalized using the language of natural systems introduced by Baues and ~\cite{BAUES1985187}. Although Hochschild cohomology is not fully functorial, 
diagrammatic Hochschild cohomology yields a natural system $\cHH^q(\cA,\cA)\colon \Tw\bC\to \mathbf{Ab}$, that is, a functor on the twisted arrow category~$\Tw \bC$ (see also Proposition~\ref{lemma:natsystHH}). Stemming from these considerations,   the
main objectives of this paper are threefold: first, to review and refine the foundational results linking Gerstenhaber-Schack complexes to Baues-Wirsching cohomology of categories; second, to study the behavior of Hochschild cohomology in the special case of homological epimorphisms of algebras; and third, to apply these results to the context of filtrations of incidence algebras of simplicial complexes. 

The relation between diagrammatic Hochschild cohomology and Baues-Wirshing cohomology  can be summarized as follows -- see~~\cite{robinson2008cohomology} and Theorem~\ref{thm:spseqtoHH}:

\begin{thm}\label{thm:intro1}
    Let $\bC$ be a finite category and $\cA$  a presheaf of $k$-algebras on $\bC$ over a field~$k$.
    Then, there is a first quadrant spectral sequence:
    \[
E_2^{p,*}\cong H^p_{BW}(\bC, \cHH^*(\cA,\cA)) 
    \]
    converging to the  diagrammatic Hochschild cohomology of the diagram $\cA$, whose rows are the Baues-Wirshing cohomology groups of $\bC$ with coefficients in the natural system $\cHH^*(\cA,\cA)$. 
\end{thm}

The isomorphisms in Theorem~\ref{thm:intro1} already have some interesting advantages. For example, assume that $\bC$ is a free finite category. Then,  Theorem~\ref{thm:intro1} implies that the second page of the considered spectral sequence is trivial when $p\geq 2$ -- cf.~Corollary~\ref{cor:highercolumns}. As a consequence,   the whole information of the spectral sequence is contained in the columns of the second page corresponding to $p=0,1$.  
To further simplify the computations on the second page of the spectral sequence, this paper restricts to the case of homological epimorphisms of algebras~\cite{GEIGLE1991273, 10.21099/tkbjm/1496165029}. 
The reason why we focus on homological epimorphisms, rather than more general algebra maps, stems from the observation that Hochschild cohomology (which is generally not functorial in both variables) yields a functor when restricted to the category $\Alg^{\Hepi}_k$ of algebras and homological epimorphisms. As a consequence, when restricting to diagrams of algebras and homological epimorphisms, Theorem~\ref{thm:intro1} can be refined:  in Corollary~\ref{cor:ssseqwithlims} we show that the rows $E_2^{*,q}$ are isomorphic to the functor cohomology groups of the category $\bC$ (rather than of the twisted arrow category~$\Tw\bC$). As a direct consequence, we get the following (see also  Corollary~\ref{cor:spseqterminal}):

\begin{thm}
    Let $\bC$ be a category with a terminal object $\bar{c}$, and $\cA$  a presheaf of $k$-algebras on $\bC$ which is a surjective homological epimorphism.
    Then, for all $p,q \in \N$, we have 
    \[
    E_2^{p,q}=
    \begin{cases}
    0 &\text{ if } p>0\\
    \HH^q(\cA(\bar{c}),\cA(\bar{c})) &\text{ otherwise} 
    \end{cases}
    \]
    In particular, diagrammatic Hochschild cohomology of $\cA$ is the Hochschild cohomology of $\cA(\bar{c})$.
\end{thm}
In the final section of the paper, we focus on our primary application: the computation of diagrammatic Hochschild cohomology for filtrations of incidence algebras. This application might also be of independent interest in persistent homology computations~\cite{carlssondata}. In fact,  Gerstenhaber and Schack~\cite{GERSTENHABER1983143} prove that, for a simplicial complex~$\Sigma$, we have $\H^*(\Sigma;k)\cong \HH^*(I(\Sigma),I(\Sigma))$, where $I(\Sigma)\coloneqq I((\Sigma,\subseteq))$ is the incidence algebra of the face poset associated to $\Sigma$. Therefore, a filtration of simplicial complexes can be turned into a filtration of incidence algebras, preserving the cohomology groups. As simplicial injective maps of simplicial complexes induce homological epimorphisms of incidence algebras, which are also epimorphisms, we can apply the developed theory to this case. Then, we show the following:
\begin{thm}
    Let $\mathcal{I}\colon [n]^{\op}\to \mathbf{Alg}_k$ be induced by a filtration of finite simplicial complexes, via $\mathcal{I}(i)\coloneqq I(\Sigma_i)$ for $0\leq i\leq n$, and restriction maps. Then, we have
    \[
    \HH_{\GS}^q(\mathcal{I}, \mathcal{I})\cong \H^q(\Sigma_n;k)
    \]
    for all $q\geq 0$. 
\end{thm}

To summarize, in this paper, we study a spectral sequence that converges to diagrammatic Hochschild cohomology, explicitly describing the cohomology groups in terms of higher limits of functors over twisted arrow categories. Then, we simplify the computation of the second page of such a spectral sequence using homological epimorphisms of algebras. Finally, we apply these computations to the case of filtrations of incidence algebras, revealing deep connections to combinatorial and topological invariants. This interplay between theory and application not only generalizes existing results but also paves the way for future developments in applied algebraic topology. 

\subsection*{Organization of the paper}
Section~\ref{section:basics} introduces classical notions of Hochschild cohomology and functor (co-)homology. Section~\ref{section:diagrammaticHH} develops the diagrammatic perspective via the Gerstenhaber-Schack framework and its reinterpretation through natural systems. Section~\ref{section:homologicalepi} explores diagrams of homological epimorphisms and their implications in the computation of Hochschild cohomology. Section~\ref{sectionapplications} applies these results to incidence algebras, highlighting their relevance to persistent homology. 

\subsection*{Conventions}

Throughout the paper, if not otherwise specified, $k$ shall denote a field. All $k$-algebras are finite-dimensional, and, unless otherwise specified, all categories are finite.

\section{Preliminary notions}
\label{section:basics}

In this section, to set the notations and basic notions needed in the follow-up, we first recall the definition of Hochschild cohomology (see, eg.~\cite{loday}, for a more detailed account), the base-change spectral sequence~\cite[Chapter~XVI.5]{cartanEilenberg} together with its role in Hochschild cohomology computations, and then we proceed by reviewing the notion of functor (co-)homology of categories~\cite{Gabriel1967CalculusOF}. 

\subsection{Hochschild cohomology}

Let $k$ be a field, $A$ an associative $k$-algebra and $M$ a $A$-bimodule. Consider the cochain complex consisting of the cochains
\begin{equation}\label{eq:HH}
    C_{\HH}^n(A,M)\coloneqq \{ f\colon \prod_n A\to M \mid f \text{ is $n$-multilinear} \} = \Hom_k(A^{\otimes n},M) \ ,
\end{equation}
equipped with the Hochschild differentials $d_{\HH}^n\colon C_{\HH}^{n-1}(A,M)\to C_{\HH}^{n}(A,M)$ for each $n>0$. The differentials~$d_{\HH}^n$ are defined by the formula:
\[
\begin{split}
    d_{\HH}^n f(a_1,\dots,a_{n})\coloneqq & a_1 f(a_2,\dots, a_{n})  +\\
    & \sum_{0<i<n}(-1)^i f(a_1,\dots, a_ia_{i+1},\dots,a_{n})+\\
    & (-1)^{n} f(a_1,\dots,a_{n-1})a_{n} \ ,
\end{split}
\]
with $d_{\HH}^0$ to be the zero map.
The  cohomology groups $\HH^n(A,M)\coloneqq \ker d_{\HH}^{n+1}/\mathrm{Im} \; d_{\HH}^{n}$ 
associated to the cochain complex $(C^*_{\HH}(A,M), d_{\HH})$  are called the Hochschild cohomology groups of $A$ with coefficients in $M$. If $M=A$, we write $\HH^*(A)$ for the Hochschild cohomology of $A$ with coefficients in $A$ (seen as a bimodule over itself). In degree $0$, the Hochschild cohomology group $\HH^0(A)$ yields the center of the $k$-algebra $A$, and in degree~$1$ it is related to the so-called K\"ahler differentials of $A$~\cite{loday}. Observe also  that Hochschild cohomology is graded with grading $\HH^*(A,M) =\bigoplus_{n\geq 0}\HH^n(A,M)$.

In the assumption that~$k$ is a field (but, more generally, if $A$ is a projective $R$-module over a commutative unital ring~$R$), it is standard to interpret Hochschild cohomology groups in terms of Ext-functors -- see, e.g.~\cite [Chapter~IX.6]{cartanEilenberg} or \cite[Corollary~9.1.5]{Weibel}. With this interpretation at hand, we have 
\begin{equation}\label{eq:HHExt}
\HH^n(A,M)\cong\Ext^n_{A^e}(A,M) \  ,    
\end{equation}
where $A^e\coloneqq A\otimes A^{\op}$ is the enveloping algebra of $A$.  The interpretation of Hochschild cohomology in terms of Ext-groups allows one to use, in computations, any projective resolution of $A$ as an $A^e$-module.

   \begin{example}\label{ex:Aproj}
       Let $A$ be an $A^e$-projective module. Then, $\HH^n(A,M)=0$ for all $n>0$ and all $A$-bimodules $M$.  
   \end{example}

   Example~\ref{ex:Aproj} can be improved. In fact,  if $\HH^n(A,M)=0$ for all $n>0$ and all $A$-bimodules~$M$, then $A$ is $A^e$-projective and, equivalently, $A$ is a separable $k$-algebra. 

\begin{example}
    Let  $A=M_p(k)$ be the $k$-algebra of $p\times p $ matrices over the field $k$. Then, $A$ is $A^e$-projective, and, by Example~\ref{ex:Aproj}, its Hochschild cohomology groups are all trivial in positive degrees. 
\end{example}
   
   Computations of Hochschild cohomology groups for non-separable algebras can generally be more difficult, and one can easily get non-trivial Hochschild cohomology groups in any positive degree. For example, a direct computation shows the following:

   \begin{example}
       Let $k=\bR$ and $A=\bR[x]/(x^m)$. Then, we have
\[
\HH^n(A)=
\begin{cases}
    A & \text{if } n=0\\
    A/(mx^{m-1}) & \text{if } n>0
\end{cases}
\]
One can generalize this example to any field $k$, and any monic polynomial $f$ in $k[x]$ -- cf.~\cite{zbMATH00028035, redondo}. 
   \end{example}

Hochschild cohomology of algebras is not well-behaved with respect to morphisms of pairs of algebras and bimodules. 
   \begin{rem}\label{rem:functo}
   Hochschild cohomology does not yield a bifunctor $(A,M)\mapsto  \HH^*(A,M)$ on the category of pairs of $k$-algebras and $ k$-algebra morphisms, and bimodules. However, it satisfies some limited functoriality, as we now describe. First, observe that Hochschild cohomology is covariant with respect to the second argument. On the other hand, if $M$ is a $A$-bimodule, then a morphism $\phi\colon B\to A$ of $k$-algebras induces a $B$-bimodule structure on~$M$. Then, restriction of scalars yields a map $\phi^*\colon \HH^*(A,M)\to \HH^*(B,M)$ between the Hochschild cohomology groups~\cite[Section~1.5]{loday}; note that $\phi^*$ is induced by the cochain map $C^*_{\HH}(A,M)\to C^*_{\HH}(B,M)$, and Hochschild cohomology is contravariant in the first argument. This latter property implies that Hochschild cohomology does not generally extend to a functor $A\mapsto \HH^*(A,A)$ on the whole category of $k$-algebras and $k$-algebra morphisms. \end{rem}

   \begin{rem}
       Although Hochschild cohomology is not generally functorial, functoriality holds for an important class of $k$-algebras; for example, restricting to the category of symmetric Frobenius algebras yields a contravariant functor (in Hochschild cohomology)~\cite{MR3219576}.
\end{rem}
   
\subsection{The base-change spectral sequence} 
Let $\Mod_R$ be the category of $R$-modules over a ring~$R$, and $\mathbf{Ab}$ the category of Abelian groups. If $R\to S$ is a ring homomorphism and $N$ is an $S$-module, there are natural functors $(-)\otimes_R S \colon \Mod_R\to \Mod_S$ and $\Hom_S(-,N)\colon \Mod_S\to \mathbf{Ab}$, and   a spectral sequence called the \emph{base-change spectral sequence}: 

    \begin{thm}[{\cite[Chapter~XVI.5]{cartanEilenberg}}]\label{thm:change_ring_ss}
        Let $\phi\colon A\to B$ be a morphism of $k$-algebras. 
        Let $M$ be a left $A$-module and $N$ a left $B$-module (considered also as $A$-module via $\phi$). Then, there is a spectral sequence of cohomological type
        \[
       E^{p,q}_2= \Ext^p_B(\Tor^A_q(B,M),N)\Rightarrow \Ext^{p+q}_A(M,N)
        \]
        converging to $\Ext^*_A(M,N)$, naturally in both $M$ and $N$. 
    \end{thm}

For a morphism $\phi\colon A\to B$ of $k$-algebras, and $N$ a given  $B$-module, the base-change spectral sequence of Theorem~\ref{thm:change_ring_ss} is the Grothendieck spectral sequence associated to the functors $\Phi=-\otimes_A B$ and $\Theta= \Hom_B(-,N)$, and it is natural in both the $A$-modules $M$ and the $B$-modules $N$.
We also refer to~\cite[Section~2]{suárezalvarez2007applications} for a more detailed description of its construction. 
    As a corollary of Theorem~\ref{thm:change_ring_ss}, we get a  spectral sequence converging to Hochschild cohomology (see also~\cite[Section~2.3.1]{suárezalvarez2007applications}): 

    \begin{cor}\label{cor:changeringspseq}
        Let $\phi\colon A\to B$ be a morphism of $k$-algebras. Then, for any $B$-bimodule $N$ there is a  spectral sequence
        \[
        \Ext^p_{B^e}(\Tor_q^{A}(B, B),N)\Rightarrow \HH^*(A,N)
        \]
        which is natural in $N$ (seen as $A$-bimodule via $\phi$).
    \end{cor}

    \begin{proof}
        The morphism $\phi\colon A\to B$ induces a morphism $\phi^e\colon A^e\to B^e$ between the associated enveloping algebras. By Theorem~\ref{thm:change_ring_ss},  for any $B$-bimodule $N$ there is a spectral sequence 
        \[
        \Ext^p_{B^e}(\Tor^{A^e}_q(B^e,A),N)\Rightarrow \Ext^*_{A^e}(A,N)
        \]
        converging to $\Ext^*_{A^e}(A,N)\cong\HH^*(A,N)$. As $k$ is a field (but, more generally, for a semi-simple ring), the graded group $\Tor_*^{A^e}(B\otimes_k B^{\mathrm{op}}, A)$ is canonically isomorphic to $\Tor_*^{A}(B, B)$ by~\cite[Corollary~IX.4.4]{cartanEilenberg}. The statement about naturality follows by naturality in $N$ of the base-change spectral sequence. 
    \end{proof}

    If we further assume that $\Tor_0^A(B, B)=B\otimes_A B$ is isomorphic to $B$ (as $B$-modules, the isomorphism being induced by the multiplication), then the edge morphism in the base-change spectral sequence gives a map 
    \begin{equation}\label{eq:spseq}
    \HH^n(B)=\Ext^n_{B^e}(B,B)\to \HH^n(A,B)    
    \end{equation}
    which is a morphism of $k$-algebras (when both the Hochschild cohomology and Ext groups are equipped with the respective cup products) -- \emph{cf}.~\cite{suárezalvarez2007applications}. Note also that, if all groups $\Tor^A_q(B,B)$ vanish for $q\geq 1$, then the base-change spectral sequence degenerates at the second page, collapsing on the $p$-axis. In such case, the homology of the total complex, as well as the groups $\HH^*(A,N)$, can be identified with the groups~$E_{\infty}^{*,0}=E_2^{*,0}=\Ext^*_{B^e}(\Tor_0^A(B,B),N)$.

\subsection{Functor (co-)homology }

Let $\mathbf{C}$ be a small category and $\bA$ a complete and cocomplete Abelian category. 
Given a functor $\cF\colon \mathbf{C} \to \bA$,   we  recall  the definitions of \emph{functor homology} (resp.~\emph{cohomology}) groups $\mathrm{H}_*(\bC;\cF)$ (resp.~$\mathrm{H}^*(\bC;\cF)$) of $\bC$ with coefficients in $\cF$. We follow \cite{Gabriel1967CalculusOF} and  \cite[Section~6.1]{primo}. To do so, we first define functor (co-)homology in terms of Kan extensions, and then provide a more direct, computable approach.

Let~$\mathbf{1}$ be the category with a single object and a single morphism. There is a unique functor~$\mathcal{T}\colon \mathbf{C}\to \mathbf{1}$. Since $\bA$ is complete and cocomplete, both left and right Kan extensions of $\cF$ exist. In particular, the left Kan extension~$\Lan_{\mathcal{T}}\cF$ of $\cF$ along~$\mathcal{T}$ exists, and it yields the colimit functor of~$\cF$. Then, the \emph{functor homology}~$\mathrm{H}_n(\mathbf{C};\cF)$ of $\mathbf{C}$ with coefficients in~$\cF$ is  the $n$-th left derived functor of~$\Lan_{\mathcal{T}}\cF$; see, eg.~\cite{maclane:71}. We denote by $\colim_n\cF$ such derived functors, which are also called the \emph{higher colimits} of~$\cF$.
Analogously, the right Kan extension along~$\mathcal{T}$ exists and it yields the limit of~$\cF$; thus, functor cohomology~$\mathrm{H}^n(\mathbf{C};\cF)$ is given by the $n$-th derived functor of~$\lim \cF$. We denote them by $\lim^n\cF$, and these are also called the \emph{higher limits} of $\cF$ -- see  \cite{Gabriel1967CalculusOF} for the equivalence of the two approaches.

 Functor homology groups can be computed, more explicitly, as the homology groups of the chain complex:
\[
   \dots \xrightarrow{\partial_{n}} \bigoplus_{c_0\to \dots\to c_n} \cF(c_0) \xrightarrow{\partial_{n-1}}  \dots \xrightarrow{\partial_2} \bigoplus_{c_0\to c_1\to c_2} \cF(c_0) \xrightarrow{\partial_1} \bigoplus_{c_0\to c_1} \cF(c_0) \xrightarrow{\partial_0} \bigoplus_{c_0\in \mathbf{C}} \cF(c_0) \to 0
\]
indexed on the chains $c_0\to\dots\to c_n$ of objects and morphisms of $\bC$. This chain complex is equipped with the differential
\[
\partial_{n}(f(c_0\to\dots\to c_{n+1}))=\cF(c_0\to c_1)f(c_1\to\dots c_{n+1})+ \sum_{i=1}^{n+1} (-1)^i f(c_0\to\dots\to \widehat{c_i}\to \dots c_{n+1})
\]
where $\widehat{c_i}$ means that the object $c_i$ is removed from the chain, and where  $(c_0\to \dots\to c_n)$ denotes the inclusion of~$f\in \mathcal{F}(c_0)$ into the summand corresponding to the sequence $c_0\to\dots\to c_{n+1}$.

Dually, the Roos complex~\cite{Roos179864} computes the functor cohomology groups $\mathrm{H}^n(\mathbf{C};\cF)$. This is the cochain complex
\[
0\to \prod_{c_0\in \mathbf{C}} \cF(c_0)\xrightarrow{d^0} \prod_{c_0\to c_1} \cF(c_1)\xrightarrow{d^1} \prod_{c_0\to c_1\to c_2} \cF(c_2)\xrightarrow{d^2}\dots\xrightarrow{d^{n-1}} \prod_{c_0\to \dots\to c_n} \cF(c_n) \xrightarrow{d^{n}}\dots
\]
 with the differential $d^n$, whose evaluation on~$f\in \prod_{c_0\to \dots\to c_n} \cF(c_n)$ is given by
\begin{align*}
d^n(f)(c_0\to\dots\to c_{n+1})= &\ (-1)^{n+1}\cF(c_n\to c_{n+1})f(c_0\to \dots\to c_n)+ \\
 + &\sum_{i=0}^n (-1)^i f(c_0\to\dots\to \widehat{c_i}\to \dots c_{n+1}) \ .
\end{align*} 
In the previous formula, the symbol $(c_0\to \dots\to c_n)$ denotes the projection onto the factor corresponding to the sequence~$c_0\to\dots\to c_{n+1}$.

\begin{rem}
    Similar constructions can be obtained using contravariant functors instead of covariant ones -- hence, functors on the opposite category $\bC^\op$ of $\bC$. Note that passing to the opposite category implies that the Roos complexes' description should be modified accordingly. For example, if $\cF$ is contravariant on $\bC$, then the Roos complex computing the functor cohomology of $\bC$ with coefficients in $\cF$ is 
    \[
0\to \prod_{c_0\in \mathbf{C}} \cF(c_0)\xrightarrow{d^0} \prod_{c_0\to c_1} \cF(c_0)\xrightarrow{d^1} \prod_{c_0\to c_1\to c_2} \cF(c_0)\xrightarrow{d^2}\dots\xrightarrow{d^{n-1}} \prod_{c_0\to \dots\to c_n} \cF(c_0) \xrightarrow{d^{n}}\dots
\]
and, analogously, for functor homology. For further reference about this approach, see~\cite{turner-everitt-cell}.
\end{rem}

The homology of a category with coefficients in a functor has been extensively studied, and the literature is vibrant. When restricted to constant functors, the functor (co)homology groups depend only upon the geometric realization of the category $\bC$ -- cf.~\cite{quillenI}. 
In particular, by \cite[Corollary~2]{quillenI}, if $\bC$ is a category with an initial element, then functor (co)homology with coefficients in the constant functor is trivial and isomorphic to the (simplicial) homology of a point. More generally, we have:

\begin{rem}\label{rem:fucntorterminal}
    Let $R$ be a commutative ring, $\bC$ a category with initial object $c_0$, and  $\cF\colon \bC\to \mathbf{Mod}_R$ a functor. Then,  
    \[
\mathrm{H}^n(\mathbf{C};\cF)=
\begin{cases}
    0 & \text{if } n\geq 1\\
    \cF(0) & \text{if } n=0
\end{cases} \ .
\]
    For a proof of this fact, see, for example, \cite[Corollary~3.13]{functor_terminal} or also~\cite[Lemma~1 (3)]{everitt-turner-Booleancovers}. Analogously, if $\cF $ is a contravariant functor and $\bC$ has a terminal object $T$, then we have $\mathrm{H}^*(\mathbf{C};\cF)=0 $ for $*\geq 1$ and $\mathrm{H}^0(\mathbf{C};\cF)=\cF(T)$. 
\end{rem}

The cohomology of categories with coefficients in a functor was generalized by Baues and Wirsching in~\cite{BAUES1985187}, where they consider more general coefficient systems. We recall their construction in Section~\ref{sec:natsyst}. To avoid confusion between the two notions, sometimes we shall denote the functor cohomology of $\bC$ with respect to a functor $\cF$ simply by $\lim^n_{\bC} \cF$. 

\section{Diagrammatic Hochschild cohomology}
\label{section:diagrammaticHH}

In this section, we recall the definition of Hochschild cohomology of diagrams of $k$-algebras, or, shortly, diagrammatic Hochschild cohomology, as developed by Gerstenhaber and Schack (and for this reason sometimes referred to as \emph{Gerstenhaber-Schack cohomology}); for further details, we refer to the various texts~\cite{zbMATH03865538, GERSTENHABER1983143, MR0981619, MR0965749} and to the references therein. Then, we proceed by describing diagrammatic Hochschild cohomology in terms of Baues-Wirsching cohomology of categories~\cite{BAUES1985187}. Although most of the discussion in this section holds for locally finite categories and, more generally, for small categories, we shall restrict to the case of finite categories.

\subsection{Gerstenhaber-Schack complex} 
For a {finite}  category $\bC$, let $\cA$ be a presheaf of $k$-algebras on $\bC$. By a presheaf of $k$-algebras, we mean a functor 
\[
\cA\colon \bC^\op\to \mathbf{Alg}_k
\]
from the opposite category associated to $\bC$ to the category of $k$-algebras, and $k$-algebra homomorphisms. 
A presheaf of $\cA$-bimodules is a functor~$\cM$ on $\bC^\op$ such that~$\cM(c)$ is  a $\cA(c)$-bimodule for all $c\in \bC$, and such that for each morphism $c\to d$ in $\bC$ the map $\cM(d)\to \cM(c)$ is a morphism of $\cA(d)$-modules -- where $\cM(c)$ is viewed as a $\cA(d)$-module via the map $\cA(d)\to \cA(c)$.

Let $\Nerve(\bC)$ be the nerve of the category $\bC$, that is, the simplicial complex on the objects of $\bC$ with $p$-simplices given by the paths of morphisms in $\bC$ of length $p+1$. If $\sigma\in \Nerve(\bC)_p$ is a $p$-simplex,  we  write $\sigma$ as a chain $c_0\to c_1\to \dots\to c_p$ of elements of $\bC$; for such a simplex $\sigma$ we set $\max \sigma \coloneqq c_p$ and $\min\sigma\coloneqq c_0$. Using the definition of the Hochschild cochain complex of Equation~\eqref{eq:HH}, we construct a doubly graded $k$-module
\begin{equation}\label{eq:doublechaincomplexHH}
C^{p,q}(\cA,\cM)\coloneqq \prod_{\sigma\in\Nerve(\bC)_p} C_{\HH}^q (\cA(\max\sigma),\cM(\min\sigma)) \ ,
\end{equation}
where the indices $p,q$ run over all non-negative integers. 

\begin{rem}
The $k$-module $C^{p,q}(\cA,\cM)$ can be seen as the module of functions $\Gamma$ on $\Nerve(\bC)_p$ associating to a $p$-simplex~$\sigma$ the cochain $\Gamma^\sigma\in C^{q}_{\HH}(\cA(\max\sigma,\cM\min\sigma)$.    
\end{rem}

The bigraded module $C^{*,*}(\cA,\cM)$ is equipped with a vertical differential, the Hochschild differential. 
There is a further horizontal differential $C^{*,q}(\cA,\cM)\to C^{*+1,q}(\cA,\cM)$ which is induced by the simplicial differential on the nerve~$\Nerve(\bC)$. More precisely, for a $p$-simplex $\sigma\coloneqq c_0\to \dots\to c_p$ and any index $r=0,\dots, p$, we set 
\[
\partial_r\sigma\coloneqq c_0\to \dots \to \widehat{c_r}\to \dots \to c_p
\]
and define $\partial\sigma\coloneqq \sum_r(-1)^r\partial_r$. 
Observe that, for $r=1,\dots, p-1$,  we have $\min\partial_r\sigma=\min\sigma$ and $\max\partial_r\sigma=\max\sigma$ by definition. Now, let $\phi$ be the algebra morphism $\phi\colon \cA(c_p)\to \cA(c_{p-1})$ induced by the morphism $c_{p-1}\to c_p$. Then, the cochain $\Gamma^{\partial_{p}\sigma}\phi$ defined by \[(\Gamma^{\partial_{p}\sigma}\phi)(a_1,\dots, a_q)\coloneqq \Gamma^{\partial_{p}\sigma}(\phi(a_1),\dots,\phi(a_q)) \ ,\] with $a_1,\dots,a_q\in\cA(c_p)$, is an element of $C^q_{\HH}(\cA(\max\sigma),\cM(\min\sigma))$. Analogously, if $T$ is the morphism $\cM(c_1)\to \cM(c_0)$ induced by the morphism $c_0\to c_1$, then $T\Gamma^{\partial_0} $ belongs to the cochain group $C^q_{\HH}(\cA(\max\sigma),\cM(\min\sigma))$. The induced simplicial differential $d_{\mathrm{simp}}\colon C^{p-1,q}(\cA,\cM)\to C^{p,q}(\cA,\cM)$ on the bigraded module $C^{*,*}(\cA,\cM)$ is finally defined as follows:
\begin{equation}\label{eq:diffHH}
(d_{\mathrm{simp}}\Gamma)^{\sigma}\coloneqq T\Gamma^{\partial_0\sigma}+\sum_{r=1}^{p-1}(-1)^r\Gamma^{\partial_r\sigma} +(-1)^p\Gamma^{\partial_p\sigma}\phi\ .
\end{equation}
The Hochschild differential and the simplicial differential anti-commute, hence $C^{*,*}(\cA,\cM)$, equipped with these differentials, is a double cochain complex -- cf.~\cite[Section~7]{zbMATH03865538}.

\begin{defn}
    The (diagrammatic) Hochschild  cohomology $\HH_{\GS}^*(\cA,\cM)$ of the diagram $\cA$ with coefficients in $\cM$ is the cohomology of the total complex of the double complex~$(C^{*,*}(\cA,\cM), d_{\mathrm{simp}}, d_{\HH}) $. 
\end{defn}

\begin{rem}\label{rem:spseqdoublecompl}
    If $C$ is a double cochain complex, there are two natural filtrations on the total complex $\Tot (C)$. The first filtration  is the \emph{column-wise filtration}
$F^I_p(\Tot (C))_n:= \bigoplus_{r\geq p} C^{r,n-r} $.
The second (row-wise) filtration~$F^{II}$ is the complementary one:
$F^{II}_q(\Tot(C))_n:=\bigoplus_{r\geq q} C^{n-r,r} $. Both filtrations yield spectral sequences and, if the double cochain complex is bounded, the associated spectral sequences converge to the cohomology of the total complex~$\Tot(C)$ -- see, e.g., \cite [Theorem~2.15]{user}. 
\end{rem}

In view of Remark~\ref{rem:spseqdoublecompl}, to compute diagrammatic Hochschild cohomology groups, we can use the spectral sequences associated with the double cochain complex~$(C^{*,*}(\cA,\cM), d_{\mathrm{simp}}, d_{\HH}) $. 
\begin{notation}\label{not:spseq}
    In the follow-up, we consider the filtration obtained by first taking the Hochschild cohomology differential in the $q$-direction, then the simplicial differential in the $p$-direction. Unless otherwise specified, we shall denote the spectral sequence obtained this way with $\{IE_r^{*,*}\}_r$. 
\end{notation}

\subsection{Diagrammatic Hochschild cohomology via natural systems}\label{sec:natsyst}
In this subsection, we show that the second page of the spectral sequence $\{IE_r^{*,*}\}_r$  converging to the Hochschild cohomology $\HH_{\GS}^*(\cA,\cM)$ of a diagram~$\cA$ on a finite category~$\bC$, with coefficients in $\cM$, has a description in terms of Baues-Wirsching cohomology of categories. We begin by recalling the concept of a natural system. 

    Let $\bC$ be a finite category, $\cA$ a presheaf of $k$-algebras on $\bC$ and $\cM$ a $\cA$-bimodule. Observe that, due to the lack of functoriality of Hochschild cohomology (see  Remark~\ref{rem:functo}), the presheaves $\cA$ and~$\cM$ {do not} generally yield a local system of coefficients on $\bC$. In fact,  for given simplices $\tau\subseteq \sigma $ in the nerve of~$\bC$, {there is no obvious} functorial map  
\[
C_{\HH}^q(\cA(\max \sigma ),\cM(\min \sigma)) \to C_{\HH}^q(\cA(\max \tau ),\cM(\min \tau))
\]
induced by $\cA$ and  $\cM$. As a consequence, we can not directly describe the spectral sequence $\{IE_r^{*,*}\}_r$ associated with the double complex ~$(C^{*,*}(\cA,\cM), d_{\mathrm{simp}}, d_{\HH}) $ in terms of homology with twisted coefficients. 
 However, the complexes $
C_{\HH}^q(\cA(\max - ),\cM(\min -)) $ yield  a \emph{natural  system}~\cite{BAUES1985187} on~$\bC$, as we now shall explain. 

For a small category $\bC$, the twisted arrow category $\Tw\bC$ is the category with objects the set of morphisms of $\bC$ and whose $\Hom$-sets $\Hom_{\Tw\bC}(f,g)$ are given by the sets of pairs $(\alpha,\beta)$ of morphisms of $\bC$ making the diagram
 \begin{center}
 \begin{tikzcd}
d \arrow[r, "\alpha"] & d' \\
 c\arrow[u, "f"]  & c' \arrow[l, "\beta"] \arrow[u, "g"']
 \end{tikzcd}    
 \end{center}
 commute. The composition is defined by juxtaposition:
 \[
 (\alpha',\beta')\circ(\alpha,\beta)=(\alpha'\circ\alpha,\beta\circ\beta') \ .
 \]
Note that there is a bifunctor $\Tw\bC\to\bC^{\op}\times \bC$ given by sending $f\colon c\to d$ to $(c,d)$.

\begin{defn}[\cite{BAUES1985187}]
   A \emph{natural system} on a category $\bC$ is a functor $F\colon \Tw\bC\to \mathbf{Ab}$.
\end{defn}

When the morphisms $\alpha$ or $\beta$ above are the identities of $\bC$, we set $\alpha_*\coloneqq F(\alpha,1)$ and $\beta^*\coloneqq F(1,\beta)$. These are the homomorphisms induced by pre- and post-composition with each arrow $f$ of $\bC$.

Given a natural system $F$ on a category $\bC$, we define the cohomology groups $H_{BW}^*(\bC,F)$ of $\bC$ with coefficients in $F$ as follows. The group of $n$-cochains is given by
\[
C^n_{BW}(\bC,F)\coloneqq \prod_{{c_n\xrightarrow{f_n} c_{n-1}\xrightarrow{f_{n-1}} \dots \xrightarrow{f_1} c_0}}F(f_n\circ f_{n-1}\circ\dots\circ f_1)
\]
which we can think of as the values of $F$ at the elements of the nerve of $\bC$, and we can think of its elements as functions~$\Gamma$. The coboundary $\delta_{BW}$ is defined by the formula
\begin{align*}
(\delta_{BW} \Gamma)(f_1,\dots,f_n)=& {f_1}_*\Gamma(f_2,\dots,f_n)+\sum_{i=1}^{n-1}(-1)^i\Gamma(f_1,\dots,f_{i}f_{i+1},\dots,f_n)+\\ & + (-1)^n {f_n}_*\Gamma(f_1,\dots,f_{n-1}) \ .
\end{align*}
We refer to \cite{BAUES1985187} for more details on the construction, along with a description of its properties.

\begin{defn}
    The (\emph{Baues-Wirshing}) cohomology $H^*_{BW}(\bC;F)$ of ~$\bC$ with values in the natural system~$F$ on $\bC$ is the cohomology of the cochain complex $(C^*_{BW}(\bC,F),\delta_{BW})$.
\end{defn}

We now go back to diagrammatic Hochschild cohomology. Let $\bC$ be a finite category, $\cA$ a presheaf of $k$-algebras on $\bC$ and $\cM$ a $\cA$-bimodule.
For each morphism $f\colon c\to d$  of $\bC$, we define 
\begin{equation}\label{eq:eqonobjs}
    \cC^q_{\HH}(\cA,\cM)(f)\coloneqq C^q_{\HH}(\cA(d),\cM(c)) \ .
\end{equation} 
If $g\colon c'\to d'$ is another morphism in $\bC$ and $(\alpha,\beta)\colon f\to g$ is a morphism in $\Tw\bC$,  in view of Remark~\ref{rem:functo} we get a natural map 
\begin{equation}\label{eq: compnatssys}
    \cC^q_{\HH}(\cA,\cM)(\alpha,\beta)\colon\cC^q_{\HH}(f)=C^q_{\HH}(\cA(d),\cM(c))\to C^q_{\HH}(\cA(d'),\cM(c'))=\cC^q_{\HH}(g)
\end{equation}
given by scalar restrictions via $\alpha$ and $\beta$.

\begin{lemma}\label{lemma:natsystHH}
    For each $q\in \N$, $\cC^q_{\HH}(\cA,\cM)$ is a natural system on $\bC$. 
\end{lemma}
\begin{proof}
    For each $q\in \N$,  $\cC^q_{\HH}(\cA,\cM)$ is defined on the objects and morphisms of the twisted arrow category of $\bC$ by the Equations~\eqref{eq:eqonobjs} and \eqref{eq: compnatssys}. When applied to an identity morphism $(1_{c},1_c)$ of $\Tw\bC$, by Eq.~\eqref{eq: compnatssys}, the induced map $\cC^q_{\HH}(\cA,\cM)\colon \cC^q_{\HH}(\cA,\cM)(1_c)\to \cC^q_{\HH}(\cA,\cM)(1_c)$ is the identity morphism of the group~$C^q_{\HH}(\cA(c),\cM(c)) $. Furthermore,  $\cC^q_{\HH}(\cA,\cM)$ preserves the compositions in $\Tw\bC$ by Eq.~\eqref{eq: compnatssys}. Therefore,  $\cC^q_{\HH}(\cA,\cM)$ yields a functor $\cC^q_{\HH}(\cA,\cM)\colon\Tw\bC\to \Ab$ on  $\Tw\bC$; which means, a natural system on $\bC$.
\end{proof}

A direct comparison of the definitions gives the following proposition.

\begin{prop}\label{prop:agreementchain}
Let  $\cA$ be a presheaf of $k$-algebras on $\bC$ and $\cM$ a presheaf of $\cA$-bimodules.
    Then, 
    \[
    (C^{q,*}(\cA,\cM),d_{\mathrm{simp}})\cong (C^*_{BW}(\bC,\cC^q_{\HH}(\cA,\cM)), \delta_{BW}) 
     \ ,
    \]
  for all $q\in \N$.
\end{prop}

\begin{proof}
    Let $q$ be a natural number. By Lemma~\ref{lemma:natsystHH}, $\cC^q_{\HH}(\cA,\cM)$ is a natural system on $\bC$. The $n$-cochains of $(C^*_{BW}(\bC,\cC^q_{\HH}(\cA,\cM)), \delta) $ are precisely the $n$-cochains of the cochain complex $(C^{q,*}(\cA,\cM),d_{\mathrm{simp}})$ defined in Eq.~\eqref{eq:doublechaincomplexHH}. A comparison of the differential $\delta_{BW}$ with the differential of Eq.~\eqref{eq:diffHH} shows that the two differentials act the same way on the cochains, yielding the required identification.  
\end{proof}

Note that the identification of Proposition~\ref{prop:agreementchain} preserves the cochain complexes' structure also as graded vector spaces over the base field $k$, but we shall not use this fact. 

\begin{rem}
    Given   a category~$\bC$,   a presheaf of $k$-algebras $\cA$ on $\bC$ and $\cM$ a presheaf of $\cA$-bimodules, we can also define the functor
\begin{equation}\label{eq:natusystHH}
\cHH^q(\cA,\cM)\colon \Tw\bC\to \mathbf{Ab}    
\end{equation}
which associates to a morphism $f\colon c\to d$ of $\bC$ the Abelian group  $\HH^q(\cA(d),\cM(c))$. By Remark~\ref{rem:functo}, we can extend the definition of $\cHH^q(\cA,\cM)$ from the objects of $\Tw\bC$ to its morphisms, 
as in Eq.~\eqref{eq: compnatssys}. Then, also $\cHH^q(\cA,\cM)$ yields a natural system on $\bC$.
\end{rem}

Let $IE_1^{*,*} $ be the first page of the spectral sequence considered in Notation~\ref{not:spseq}, that is the spectral sequence obtained from the double complex $(C^{*,*}(\cA,\cM), d_{\mathrm{simp}}, d_{\HH}) $ after taking first the homology with respect to the vertical (Hochschild) differential. 
By Proposition~\ref{prop:agreementchain}, we get an identification 
\begin{equation}\label{eq: firstpage}
(IE^1_{q,*},d_{\mathrm{simp}})=(C^*_{BW}(\bC,\cHH^q(\cA,\cM)), \delta_{BW})  
\end{equation}
between the rows of the first page of the spectral sequence $IE^{*,*}$ and the Baues-Wirshing cochain complex of $\bC$ with coefficients in the natural system $\cHH^q(\cA,\cM)$.
As a consequence, we have a description of the rows of the second page $IE_2^{*,q}$ of such spectral sequence in terms of the Baues-Wirshing cohomology groups of $\bC$, with coefficients in the natural systems $\cHH^q(\cA,\cM)$. We can summarize it as follows:

\begin{prop}\label{prop:jfj}
Let $\bC$ be  a small category, $\cA$  a presheaf of $k$-algebras on $\bC$ and $\cM$ a presheaf of $\cA$-bimodules.
    Then, for all $p,q\in \N$, the cohomology groups appearing in the second page of the spectral sequence $IE_2^{*,*}$  are given by Baues-Wirshing cohomology groups:
\[
    IE_2^{p,q}\cong H^p_{BW}(\bC; \cHH^q(\cA,\cM)) \ .
    \]
    Hence, they can be computed as the cohomology of $\bC$ with coefficients in the natural system $\cHH^q(\cA,\cM)$.    
\end{prop}

To the authors' knowledge, Proposition~\ref{prop:jfj} was  first observed in \cite{robinson2008cohomology} (see also the thesis~\cite{robinson2011cohomologylambdaringspsirings} 
and, in particular, Section~5.7) and, more recently, albeit in a different form, it was used in \cite{mundinger2023hochschild}. In view of the properties of Baues-Wirshing cohomology groups, we can further identify the rows of the second page of the spectral sequence $IE^{*,*}$ with the more classical functor cohomology groups. In fact,  by \cite[Theorem 4.4 and Remark~8.7]{BAUES1985187}, we have the isomorphisms
\begin{equation}\label{eq:isonat=fun}
H^n_{BW}(\bC,F)\cong \Ext_{\Tw\bC}^n(\bZ,F)\cong \mathrm{lim}_{\Tw\bC}^n F    
\end{equation}
between the cohomology of $\bC$ with coefficients in a natural system $F$, the Ext-groups, and the functor cohomology groups of the category $\Tw\bC$ with coefficients in $F$ (seen as a functor on $\Tw\bC$). Therefore, in view of Proposition~\ref{prop:jfj} we get the following:

\begin{thm}\label{thm:spseqtoHH}
    Let $\bC$ be  a finite category, $\cA$  a presheaf of $k$-algebras on $\bC$ and $\cM$ of $\cA$-bimodules.
    Then, there is a first-quadrant spectral sequence
    \[
    IE_2^{p,q}\cong\mathrm{lim}_{\Tw\bC}^p \cHH^q(\cA,\cM)
    \]
    converging to the diagrammatic Hochschild cohomology $\HH_{\GS}^*(\cA,\cM)$. 
\end{thm}

\begin{proof}
Consider the spectral sequence  $\{IE_r^{*,*}\}_r$ from Notation~\ref{not:spseq}. 
By construction, it converges to $\HH_{\GS}^*(\cA,\cM)$ and its first page is identified with the cochain complex associated to $\bC$ with coefficients in the natural system $\cHH^q(\cA,\cM)\colon \Tw\bC\to \mathbf{Ab}$, as by Eq.~\eqref{eq: firstpage}. Then, for each $q\in \N$, application of the simplicial differential yields an identification of the rows of the second page $IE_2^{*,q}$ with the cohomology of $\bC$ with coefficients in  $\cHH^q(\cA,\cM)$. The statement 
now follows from Eq.~\eqref{eq:isonat=fun}.
\end{proof}

We provide an illustrative example. 

\begin{example}\label{ex:[n]}
Consider the category $[n]$, which is the category with objects the natural numbers $0,\dots, n$ and with a morphism $i\to j$ if and only if $i\leq j$. A presheaf of $k$-algebras on $[n]$ associates to each number~$i$, with $i\leq n$, a $k$-algebra~$A_i$ and for $i\leq j$  a morphism of $k$-algebras $A_j\to A_i$. Therefore, for each $i\leq j$, we have $\cHH^q(\cA,\cA)(i\to j)=\HH^q(A_j,A_i)$. Consider indices $i', j'$ with $i'\leq i$ and $j\leq j'$;  then, we get a morphism 
 \begin{center}
 \begin{tikzcd}
j \arrow[r, "j\leq j'"] & j' \\
 i\arrow[u, "i\leq j"]  & i' \arrow[l, "i'\leq i"] \arrow[u, "i'\leq j'"']
 \end{tikzcd}    
 \end{center}
in the twisted arrow category $\Tw[n]$. The natural system $\cHH^i(\cA,\cA)$ gives
a map
$$\cHH^q(\cA,\cA)((i,j)\to(i',j'))\colon\HH^q(A_j,A_i)\to \HH^q(A_{j'},A_{i'})$$
(recall that Hochschild cohomology is contravariant in the first argument and covariant in the second one, and that $\cA$ is contravariant on $[n]$) between the Hochschild cohomology groups. Furthermore, this map is induced by scalar restrictions. In the case $n=2$, it yields the following diagram of $k$-algebras (where all compositions have been omitted):
 \begin{center}
 \begin{tikzcd}
& & \HH^i(A_2,A_0) & &  \\
& \HH^i(A_1,A_0) \arrow[ur]&  & \HH^i(A_2,A_1) \arrow[ul] &  \\
\HH^i(A_0,A_0) \arrow[ur, ""]& & \HH^i(A_1,A_1)\arrow[ul] \arrow[ur] & &\HH^i(A_2,A_2)\arrow[ul]
 \end{tikzcd}    
 \end{center}
By Theorem~\ref{thm:spseqtoHH}, we have   $ IE_2^{p,q}=\mathrm{lim}_{\Tw[n]}^p \cHH^q(\cA,\cA)$, hence the $q$-th line of the second page of such spectral sequence are the higher limits of the natural system $\cHH^q(\cA,\cA)$.
    \end{example}

The category $[n]$ of Example~\ref{ex:[n]} is a \emph{free category}, which means it is a category freely generated by a directed graph. For such categories, there is a bound on the (Baues-Wirshing) cohomological dimension. We first need a definition:

\begin{defn}
    Let $\bC$ be a small category. The (possibly infinite) natural number
    \[
    \dim \bC\coloneqq \min\{N\mid H^k_{BW}(\bC;D)=0  \text{ for all } k> N \text { and all natural systems } D\} 
    \]
is called the    \emph{dimension} of $\bC$.
\end{defn}

The dimension of the category is related to its homological properties. In particular, free categories have vanishing higher Baues-Wirshing cohomology; in fact, if $\bC$ is a free category, then  $\dim\bC\leq 1$ by~\cite[Theorem~6.3(A)]{BAUES1985187}.

\begin{cor}\label{cor:highercolumns}
    Let $\bC$ be a free category. Then, 
     $
    IE_2^{p,q}    =0 
    $
    for all $p\geq 2$ and $ q\geq 0$.
\end{cor}

\begin{proof}
    By assumption, $\bC$ is a free category, and by \cite[Theorem~6.3(A)]{BAUES1985187} its dimension is at most~$1$. Therefore,  $H^p_{BW}(\bC,\cHH^q(\cA,\cM))=\mathrm{lim}_{\Tw\bC}^p \cHH^q(\cA,\cM)=0 $ for all $p\geq 2$. The result follows.
\end{proof}

In view of Corollary~\ref{cor:highercolumns}, the second page of the spectral sequence abutting to the diagrammatic Hochschild cohomology is concentrated in two non-trivial columns. Therefore, computation of the Hochschild cohomology $\HH^n_{\GS}(\cA,\cM)$ results in solving the extension problems only involving the cohomology groups $\mathrm{lim}_{\Tw\bC}^0 \cHH^n(\cA,\cM)$
and $\mathrm{lim}_{\Tw\bC}^1 \cHH^{n-1}(\cA,\cM)$. 
To further simplify computations of these higher limits, we can restrict, for example, to considering presheaves of symmetric Frobenius algebras or incidence algebras. In fact, in such a case, the higher limits on the twisted arrow category reduce to higher limits on $\bC$. In the next section, we shall focus on diagrams involving homological epimorphisms.

\section{Diagrams of homological epimorphisms }
\label{section:homologicalepi}
In this section, we first recall the notion of  \emph{homological epimorphisms} of $k$-algebras, following~\cite{GEIGLE1991273, 10.21099/tkbjm/1496165029, suárezalvarez2007applications,HERMANN2016687}, and some of their properties. Then, we analyze the diagrammatic Hochschild cohomology with respect to diagrams of algebras and homological epimorphisms.

 \subsection{Homological epimorphisms} Let $\Mod_A$ be the category of finitely generated (left) modules over a $k$-algebra $A$. For a morphism $\phi\colon A\to B$ of $k$-algebras, denote by $\phi^*\colon \Mod_B\to \Mod_A$ the restriction of scalars functor, which identifies $B$-modules as $A$-modules via~$\phi$. If $M,N$ are $B$-modules, the  natural map $\Hom_B(M,N)\to \Hom_A(M,N)$ given by  restriction of scalars induces  natural homomorphisms $\Ext^i_B(M,N)\to \Ext^i_A(M,N)$ of Ext-groups. 
 Likewise, there are natural morphisms $\Hom_{A^e}(B,M)\to \Hom_{A^e}(A,M)$ and $\Ext^i_{A^e}(B,M)\to \Ext^i_{A^e}(A,M)$ for all $B$-modules $M$.

Recall that a morphism $\phi\colon A\to B$ of $k$-algebras is called an \emph{epimorphism} if it is an epimorphism in the category of rings. This is equivalent to asking that the restriction functor~$\phi^*\colon \Mod_B\to \Mod_A$ is full and faithful. A morphism $\phi\colon A\to B$ is an epimorphism if and only if the multiplication map $\Tor_0^A(B,B)=B\otimes_A B\to B$ is an isomorphism of $B$-modules~\cite[Proposition~1.1]{SILVER196744}. 

\begin{example}\label{ex:surjepi}
    A surjective map of algebras is also an epimorphism. Vice versa, finite epimorphisms are also surjective. 
\end{example}

Note that not all epimorphisms are surjective; for example, the inclusion $\,bZ\to\bQ$ is an epimorphism, but it is not surjective. 

\begin{rem}[{\cite[Corollary~1.2]{SILVER196744}}]
If $\phi\colon A\to B$ is an epimorphism, then $\phi$ maps the center of $A$ into the center of $B$.
\end{rem}

For a ring $R$, consider the homotopy category $K(R)$ of chain complexes of $R$-modules and the derived category $D(R) = D(\Mod_R)$, that is the category obtained from $K(R)$ by inverting the quasi-isomorphisms. We shall only consider bounded chain complexes, and we denote by $D^b(A)$ and  $D^b(B)$ the bounded derived categories associated to the $k$-algebras $A$ and $B$. A morphism $\phi\colon A\to B$ of $k$-algebras yields also a derived restriction of scalars functor $D^b(\phi^*)\colon D^b(B)\to D^b(A)$ between the associated derived categories.

\begin{defn}[{\cite[Definition~4.5]{GEIGLE1991273}}]
    A morphism $\phi\colon A\to B$  of $k$-algebras is a \emph{homological epimorphism} if the induced functor
    \[
    D^b(\phi)\colon D^b(B)\to D^b(A)
    \]
    is fully faithful. 
\end{defn}

The notion of homological epimorphism extends the concept of epimorphism of algebras. In fact, if $\phi\colon A\to B$ is a  epimorphism of $k$-algebras with $B$ $A$-flat, then $\phi$ is a homological epimorphism~\cite[Corollary~4.7 (1)]{GEIGLE1991273}.  There are various characterizations of homological epimorphisms also in terms of $\Ext$ and $\Tor$ functors. For the following characterization we refer to~\cite[Theorem~4.4]{GEIGLE1991273} and \cite[Lemma~2.1]{10.21099/tkbjm/1496165029};

\begin{lemma}\label{lemma:vanishingTor}
    Let $\phi\colon A\to B$ be an epimorphism of algebras. Then, the following are equivalent:
    \begin{itemize}
        \item $\phi$ is a homological epimorphism;
        \item $\phi^\op$ is a homological epimorphism;
        \item $\Tor_i^A(B,M)=0$
    for all $i\geq 1$ and all $B$-modules $M$;
    \item $\Tor_i^A(B,B)=0$
    for all $i\geq 1$;
    \item $\Ext^i_A(B,M)=0$
    for all $i\geq 1$ and all $B$-modules $M$;
    \item $\Ext^i_A(B,B)=0$
    for all $i\geq 1$;
    \item the natural morphisms 
    \[
    \Ext^i_B(M,N)\to\Ext^i_A(M,N)
    \]
    are isomorphisms for all $i\geq 1$ and all $B$-modules $M,N$.
    \end{itemize} 
\end{lemma}

As consequence of Lemma~\ref{lemma:vanishingTor}, if $\phi$ is a homological epimorphism then also $\phi^e\colon A^e\to B^e$ is a homological epimorphism -- see~\cite[Proposition~2.3 (3)]{10.21099/tkbjm/1496165029}. Furthermore, restriction of scalars along a homological epimorphism induces maps $\HH^i(B)=\Ext^i_{B^e}(B,B)\to\Ext^i_{A^e}(B,B)$ which are isomorphisms for all $i\geq 1$~\cite[Proposition~2.3 (4)]{10.21099/tkbjm/1496165029}. A more general result, first proven in \cite[Proposition~3.1]{10.21099/tkbjm/1496165029} and extended in \cite[Proposition~5.3]{suárezalvarez2007applications}, shows the following connection to  Hochschild cohomology:

\begin{prop}\label{prop:isoHH}
    Let $\phi \colon A\to B$ be an epimorphism of  $k$-algebras. If $\phi$ is also a homological epimorphism, then there is a natural isomorphism 
    \[
    \HH^*(B,-)\to \HH^*(A,-) 
    \] 
    of functors on the category of $B$-bimodules. 
\end{prop}

\begin{proof}
    In view of Lemma~\ref{lemma:vanishingTor}, the groups $\Tor_i^A(B,M)$ vanish for all positive $i$ and $B$-bimodule $M$. Furthermore, the group $\Tor_0^A(B,B)$ is isomorphic to $B$, as $B$-bimodules. Then, the spectral sequence of Corollary~\ref{cor:changeringspseq} degenerates at the second page;  all the higher differentials are trivial. Naturality of  the change of rings spectral sequences 
    yields a transformation
    \[
    \HH^n(B,-)=\Ext^n_{B^e}(B,-)\to \HH^n(A,-)    
    \]
    which is a natural isomorphism of graded modules. 
\end{proof}

Provided $A$ and $B$ are finite-dimensional $k$-algebras, the morphism between the associated Hochschild cohomology groups  in Proposition~\ref{prop:isoHH} can be described more explicitly:

\begin{prop}\label{prop:isoHH2}
    Let $\phi \colon A\to B$ be an epimorphism of finite-dimensional $k$-algebras with $\ker\phi$ a $A^e$-projective module. If $\phi$ is also a homological epimorphism, then the restriction of scalars induces isomorphisms 
    \[
    \HH^n(B,M)\to \HH^n(A,M) 
    \] 
for all $B$-bimodules $M$ and all $n\geq 0$. 
    
\end{prop}

For completeness, we include the proof, taken from \cite[Proposition~3.1]{10.21099/tkbjm/1496165029}.

\begin{proof}
   Consider the short exact sequence of $A^e$-modules
\[
0\to \ker \phi\to A\xrightarrow{\phi} B\to 0 \ .
\]
For any $B$-bimodule $M$,  apply the contravariant functor $\Hom_{A^e}(-,M)$. Then, we get a long exact sequence involving  $\Ext$-groups over $A^e$.
In particular,  $\phi$ induces  morphisms $\Hom_{A^e}(B,M)\to\Hom_{A^e}(A,M)$ and $\Ext^n_{A^e}(B,M)\to \Ext^n_{A^e}(A,M)$. We have $\Hom_{A^e}(\ker \phi, M)=0$ and  as $\ker\phi$ is a projective $A^e$-module, the groups $\Ext_{A^e}(\ker\phi,M)$ vanish. Then, by Lemma~\ref{lemma:vanishingTor}, the homological epimorphism $\phi$ induces isomorphisms $\Ext_{B^e}^n(B,M)\to \Ext_{A^e}^n(B,M) $; hence, the  morphisms \[\HH^n(B,M)=\Ext^n_{B^e}(B,M)\to \Ext^n_{A^e}(A,M)=\HH^n(A,M)\] for all $B$-bimodules~$M$ and all $n\geq 0$ follow.
\end{proof}

As a consequence of Proposition~\ref{prop:isoHH}, we get the following corollary -- see also~\cite{suárezalvarez2007applications}. 

\begin{cor}\label{cor:isoHH}
    Let $\phi \colon A\to B$ be an epimorphism of $k$-algebras. If $\phi$ is also a homological epimorphism, then it induces isomorphisms $\HH^i(B,B)\to \HH^i(A,B)$ for all $i\geq 0$. 
\end{cor}

\begin{proof}
    Apply Proposition~\ref{prop:isoHH} with $B=M$ as $B$-module.
\end{proof}

More precisely, the isomorphisms in Corollary~\ref{cor:isoHH} can be written as the composition
\begin{equation}\label{eq:extiso}
\HH^i(B)=\Ext^i_{B^e}(B,B)\to \Ext^i_{A^e}(B,B)\to \Ext^i_{A^e}(A,B)=\HH^i(A,B)
\end{equation}
where the first map is the natural isomorphism of Lemma~\ref{lemma:vanishingTor}, induced by restriction of scalars, and the second map is the morphism 
induced by $\phi$ by applying the functor $\Hom_{A^e}(-,B)$.

Let now $I$ be the kernel of a morphism $\phi\colon A\to B$ of $k$-algebras and consider the exact sequence 
\[
0\to I\to A\to B\to 0 \ .
\]
After applying the functor $\Hom_{A^e}(A,-)$, we get the long exact sequence
\[
0\to \HH^0(A,I)\to \HH^0(A)\to \HH^0(A,B)\to\HH^1(A,I)\to\HH^1(A)\to \HH^1(A,B)\to\dots
\]
and from this long exact sequence, for each $i\geq 0$, the commutative diagrams
\begin{equation}\label{eq:digrHH}
 \begin{tikzcd}
 & & & \HH^i(A,B) \arrow[dr, dotted, "\cong"]\arrow[rr]&  & \HH^{i+1}(A,I) \\
 \HH^{i}(A,I)\arrow[rr] & & \HH^i(A,A) \arrow[ur]\arrow[rr, dotted]&  & \HH^i(B,B) \arrow[ur,dotted]  &   
 \end{tikzcd}    
\end{equation}
 The dotted isomorphism is the inverse of the isomorphism in Eq.~\eqref{eq:extiso}, and the other dotted arrows are the compositions of the resulting maps, making all the triangular diagrams commute. The obtained map $\HH^i(A,A)\to\HH^i(B,B)$ also has a more familiar description, in terms of change of scalars, as we now shall explain. 
First, if $M$ is a $B$-bimodule, then a morphism $\phi\colon A\to B$ of $k$-algebras induces a $A$-bimodule structure on~$M$, which we make explicit by writing $\phi^*(M)$. 
Consider also the induced morphism $\phi^e\colon A^e\to B^e$.

 \begin{rem}\label{rem:la}
     The restriction functor ${\phi^e}^*$ along a morphism $\phi^e\colon A^e\to B^e$ has a left adjoint given by 
     \[
     B^e\otimes_{A^e} - \cong B\otimes_A - \otimes_A B \ . 
     \]
 \end{rem}

 The morphism $\phi\colon A\to B$ induces a map $\phi^*\colon \HH^*(B,M)\to \HH^*(A,\phi^*(M))$ of Hochschild cohomology groups. In view of Eq.~\eqref{eq:HHExt}, this can be written as $\Ext_{B^e}^*(B,M)\to \Ext_{A^e}^*(A,\phi^*(M))$.
 Using  the bar resolution $\mathbb{B}A$ of $A$, for each $n$, we have 
 \[
 \Ext_{A^e}^n(A,\phi^*(M))= H^n\Hom_{A^e}(\mathbb{B}A,\phi^*(M)) \ .
 \]
  Then, the adjuction  of Remark~\ref{rem:la} -- see also \cite[Lemma~8.2]{HERMANN2016687} -- induces an isomorphism
  \[
  H^n\Hom_{A^e}(\mathbb{B}A,\phi^*(M))\cong H^n\Hom_{B^e}(B\otimes_A\mathbb{B}A\otimes_A B,M)
  \]
  of cohomology groups, and, by \cite[Lemma~7.13]{HERMANN2016687}, $B\otimes_A\mathbb{B}A\otimes_A B$ yields a projective resolution of $B$ over $B^e$ -- provided $\phi$ is a homological epimorphism. As a consequence, we get an isomorphism $H^n\Hom_{B^e}(B\otimes_A\mathbb{B}A\otimes_A B,M)\cong\Ext^n_{B^e}(B,M)$. The induced map
  \[
 \Ext_{A^e}^n(A,\phi^*(M))\to  \Ext^n_{B^e}(B,M)
  \]
 yields an inverse of the isomorphism~\eqref{eq:extiso}
 when $M=B$, and the following version of Corollary~\ref{cor:isoHH}: 

\begin{lemma}[{\cite[Lemma~8.2]{HERMANN2016687}}]\label{lemma:fakfka}
    Let $A$ be a $k$-algebra, $I\subseteq A$ a two-sided ideal, and $\pi\colon A\to B$ the quotient. Assume that both $A$ and $B\coloneqq A/I$ are projective as $k$-modules. Then, if $\Tor^{A}_i(B,B)=0$ for all $i\geq 1$,  the functor
    \[
    B\otimes_A -\otimes_A B\colon \Mod(A^e)\to \Mod(B^e) \ ,
    \]
adjoint to the restriction of the scalars functor, gives rise to an isomorphism 
    \[
    \HH^*(A,B)\xrightarrow{\cong} \HH^*(B)
    \]
    of graded modules over $\HH^*(A)$.
\end{lemma}

In view of 
\cite[Theorem~8.11]{HERMANN2016687}, 
the resulting composition
\[
\HH^*(A,A)\to \HH^*(A,B)\xrightarrow{\cong} \HH^*(B,B)
\]
is induced by the derived tensor product $B\otimes_A^{\mathbb{L}} -\otimes_A^{\mathbb{L}} B$, concluding our description of the diagram \eqref{eq:digrHH}.
 
 In the assumptions of Lemma~\ref{lemma:fakfka}, we get a long exact sequence relating the Hochschild cohomology groups of $A$, $I$, and~$B$:

 \begin{cor}[{\cite[Theorem~8.11]{HERMANN2016687}}]\label{cor:lesquotients}
     Let $\pi\colon A\to A/I$ and set $B\coloneqq A/I$. If $\pi$ is a homological epimorphism, there is an induced  long exact sequence
\begin{equation}\label{eq:lesHH}
0\to \HH^0(A,I)\to \HH^0(A)\to \HH^0(B)\to\HH^1(A,I)\to\HH^1(A)\to\HH^1(B)\to\cdots
\end{equation}
of Hochschild cohomology groups. 
 \end{cor}

It is evident from Lemma~\ref{lemma:fakfka} and Corollary~\ref{cor:lesquotients} the importance of homological epimorphisms in Hochschild cohomology computations. Furthermore, homological epimorphisms are also related to the homological properties of the ideals at hand. For example, in view of Corollary~\ref{cor:lesquotients} and in the assumption that the map $\pi\colon A\to A/I$ is a homological epimorphism, to compare the Hochschild cohomologies of the algebras $A$ and $B$, we need to compute also the relative cohomology groups $\HH^*(A,I)$. We now recall the notion of homological ideal, which dates back to Auslander (see~\cite{Auslander1992HomologicalTO}, where such ideals were called \emph{strong idempotent ideals}).

\begin{defn}
    An ideal $I$ of $A$ is a \emph{homological ideal} if the quotient $\pi\colon A\to A/I$ is a homological epimorphism. 
\end{defn}

If $\phi\colon A\to B$ is a epimorphism and we set $I\coloneqq \ker\phi$, then we have $\Tor_1^A(B,B)\cong I/I^2$; hence, if $\phi$ is also a homological epimorphism, Lemma~\ref{lemma:vanishingTor} implies that $I$ is idempotent. We have a partial converse, which will be useful to us in the next section:

\begin{lemma}[{\cite[Proposition~2.2 (c)]{10.21099/tkbjm/1496165029}}]\label{lemma:homideals}
    Let $I$ be an ideal of $A$ which is a projective $A$-module, with $I/I^2=0$. Then, $\pi\colon A\to A/I$ is a homological epimorphism. 
\end{lemma}

\begin{proof}
    Take the short exact sequence
    \[
    0\to I\to A\to A/I\to 0
    \]
    and apply the functor $-\otimes_A B$. We get $\Tor_i^A(A/I,A/I)\cong \Tor_{i-1}^A(I,A/I)$ for all $i\geq 2$. On the other hand, $\Tor_1^A(A/I,A/I)\cong I/I^2$. Then, as $I$ is a projective $A$-module and $I/I^2=0$, all the $\Tor$-groups vanish.  
\end{proof}

 We now let $\Alg^{\Hepi}_k$ be the category of $k$-algebras and homological epimorphisms of algebras which are also surjections.
When $k$ is a commutative ring, we also consider the category $\Alg^{\Hepi,\mathrm{pr}}_k$, which is the full subcategory of $\Alg^{\Hepi}_k$  on algebras that are projective as $k$-modules; if $k$ is a field, the two categories coincide.   It was shown in \cite[Lemma~A.2]{arXiv:2104.06516} that Hochschild cohomology is functorial with respect to homological epimorphisms of (dg) categories, and the same arguments apply to the case of $k$-algebras, giving the following:
\begin{prop}\label{prop:coherentdiagrams}
      Let $k$ be a commutative ring. Then, for each $i\in \N$, Hochschild cohomology  
     \[
     \HH^i\colon \Alg^{\Hepi,\mathrm{pr}}_k\to \mathbf{Ab}
     \]
     yields a functor.
 \end{prop} 
 
 The primary argument relies on the fact that the extension of scalar functors, which serves as the left adjoint to the restriction of scalars, maps diagonal bimodules to diagonal bimodules through homological epimorphisms, see~\cite[Lemma~A.1]{arXiv:2104.06516} and \cite[Proposition~3.4]{zbMATH07286825}. 
   We wish to recall here that functoriality for Hochschild cohomology of categories was first, and more extensively, shown by Keller~\cite{keller2003derived,zbMATH07431618}, but more general results have also been shown more recently -- see, eg.~\cite{zbMATH07522563,arXiv:2509.14620}.

\subsection{Diagrammatic Hochschild cohomology with homological epimorphisms}

Let  $\bC$ be a small category and $\cA$ a presheaf of $k$-algebras on $\bC$.
We say that $\cA$ is surjective if all the induced maps $\cA_{c,d}\colon \cA(d)\to\cA(c)$ are surjective. Analogously, we  define homological epimorphisms for presheaves:

\begin{defn}
    The presheaf $\cA\colon \bC^\op\to \Alg_k$ is a \emph{homological epimorphism} if the induced map $\cA_{c,d}\colon \cA(d)\to\cA(c)$ is a homological epimorphism of $k$-algebras for all morphisms $c\to d$ in $\bC$.
\end{defn}

A homological homomorphism $\cA$ is a functor $\cA\colon \bC^\op\to \Alg_k^\Hepi$.
In view of 
Proposition~\ref{prop:coherentdiagrams}, 
composition of Hochschild cohomology in degree~$i$, for $i\geq 0$, with a presheaf $\cA\colon \bC^\op\to\Alg^{\Hepi}_k$ 
yields by composition a functor 
\[
\HH^i(\cA(-), \cA(-))=\HH^i\circ\cA\colon\bC^\op\to\Ab
\]
which, by slight abuse of notation, we denote by $\HH^i(\cA,\cA)$. 

\begin{thm}\label{thm:selfduality}
Let $\cA\colon \bC^\op\to\Alg^{\Hepi, \mathrm{pr}}_k$ be a surjective homological epimorphism. 
Then,  the higher limits $\mathrm{lim}_{\Tw\bC}^p \cHH^q(\cA,\cA)$ and $ \mathrm{lim}_{\bC}^p \HH^q(\cA(-),\cA(-))$ are isomorphic. In particular, we have a chain of isomorphisms
\[
H^p_{BW}(\bC,\cHH^q(\cA,\cA))\cong\mathrm{lim}_{\Tw\bC}^p \cHH^q(\cA,\cA)\cong \mathrm{lim}_{\bC}^p \HH^q(\cA,\cA) \cong H^p(\bC,\HH^q(\cA,\cA)) \ ,
\]
for all $p,q\geq 0$. 
\end{thm}

\begin{proof}
We only need to show the second isomorphism.    Let $c\to d$ be an arrow of $\bC$ and consider the morphism of $k$-algebras $\cA_{c,d}\colon \cA(d)\to\cA(c)$ induced by $c\to d$. By assumption, this is a surjective homological epimorphism. As shown in the previous section, for all $i\in \N$ we have  natural isomorphisms $  \HH^i(\cA(d),\cA(c))\cong \HH^i(\cA(c),\cA(c))$, and by Proposition~\ref{prop:coherentdiagrams},  we get  a functor 
\[
\HH^i\circ \cA\colon \bC^\op\to \Ab
\]
assigning to an object $c$ of $\bC$ the Hochschild cohomology group $\HH^i(\cA(c),\cA(c))$;  to a morphism $c\to d$ of $\bC$ it assigns the map $\HH^i(\cA(d))\to\HH^i(\cA(c))$ induced by extension of scalars.

Consider now the natural system $\cHH^q(\cA,\cA)$ of Eq.~\eqref{eq:natusystHH}. We want to show that there is a natural isomorphism of functors $\eta\colon \cHH^q(\cA,\cA)\to\pi\circ\HH^q\circ \cA$, given by the diagram
\begin{center}
 \begin{tikzcd}
\Tw\bC \arrow[rr, "{\cHH^q(\cA,\cA)}"]\arrow[d, "\pi"'] & & \Ab \\
 \bC^\op\arrow[urr, "\HH^q\circ \cA"']  & & 
 \end{tikzcd}    
 \end{center} 
 where $\pi\colon \Tw\bC\to\bC^\op$ is the standard projection. In fact, for any object $c\to d$ of $\Tw\bC$ and $q\in\N$, we have an isomorphism $\eta(c\to d)\colon \HH^q(\cA(d),\cA(c))\xrightarrow{\cong} \HH^q(\cA(c),\cA(c))$: this holds by Corollary~\ref{cor:isoHH}, and, more explicitly, by (the inverse of the map in) Lemma~\ref{lemma:fakfka}. 
For a morphism  
\begin{equation}\label{eq:mortw}
 \begin{tikzcd}
d \arrow[r, "\alpha"] & d' \\
 c\arrow[u, "f"]  & c' \arrow[l, "\beta"] \arrow[u, "g"']
 \end{tikzcd}    
 \end{equation} 
 in $\Tw\bC$, we get the commutative diagram of algebras
 \begin{center}
 \begin{tikzcd}
\cA(d) \arrow[d, "f"'] & \cA(d') \arrow[d, "g"]\arrow[l, "\alpha"']\\
 \cA(c) \arrow[r, "\beta"'] & \cA(c')  
 \end{tikzcd}    
 \end{center} 
 All maps are surjective homological epimorphisms. Recall that  Hochschild cohomology is covariant in the second entry and contravariant in the first entry. Then, consider the resulting diagram
 \begin{center}
 \begin{tikzcd}
 \HH^i(\cA(d),\cA(c)) \arrow[d, "\eta", "\cong"']\arrow[r, "\beta_*"]& \HH^i(\cA(d),\cA(c'))\arrow[r,"\alpha^*"]  & \HH^i(\cA(d'),\cA(c')) \arrow[d, "\cong", "\eta"'] \\
 \HH^i(\cA(c),\cA(c)) \arrow[r, "\beta_*"]& \HH^i(\cA(c),\cA(c'))\arrow[u,"f^*"] \arrow[r, dotted, bend right, "\cong"]& \HH^i(\cA(c'),\cA(c'))\arrow[l, "\beta^*"']
 \end{tikzcd}    
 \end{center}
 where all maps in the right square are isomorphisms by restriction of scalars. The upper composition yields the value of the functor $\cHH^i$ at the morphism \eqref{eq:mortw} of the twisted arrow category. The rightmost bottom map $\beta^*$ is inverse to the map $\HH^i(\cA(c),\cA(c'))\to \HH^i(\cA(c'),\cA(c'))$ of Lemma~\ref{lemma:fakfka}, and, such inverse, is coherent with the natural transformation $\eta$. In view of \cite[Theorem~8.11]{HERMANN2016687}, the resulting composition $\HH^i(\cA(c),\cA(c))\to \HH^i(\cA(c'),\cA(c'))$ is the map between the Hochschild cohomology groups of \cite[Theorem~7.6]{HERMANN2016687} induced by the derived tensor product. We get a commutative diagram
  \begin{center}
 \begin{tikzcd}
 \HH^i(\cA(d),\cA(c)) \arrow[d, "\cong"']\arrow[rrr]& & & \HH^i(\cA(d'),\cA(c')) \arrow[d, "\cong"]  \\
 \HH^i(\cA(c)) \arrow[rrr, "\cA(c')\otimes_{\cA(c)}^{\mathbb{L}} -\otimes_{\cA(c)}^{\mathbb{L}} \cA(c')"]& & & \HH^i(\cA(c'))
 \end{tikzcd}    
 \end{center}
 which implies that 
 the natural system $\cHH^i(\cA,\cA)$ is obtained -- up to natural isomorphisms -- by pulling back the functor $\HH^i(\cA,\cA)$ along the composition 
 \[\pi\colon \Tw\bC\to \bC^{\op}\times\bC\to \bC^\op \  . \] In such a case, Baues-Wirsching cohomology groups with coefficients in a natural system agree with functor cohomology groups (see~\cite{BAUES1985187}), and as isomorphic natural systems yield isomorphic Baues-Wirsching cohomology groups, the statement follows. 
\end{proof}

As a corollary of Theorem~\ref{thm:spseqtoHH} and Theorem~\ref{thm:selfduality}, we get the main result of the section, which clarifies the structure of the spectral sequence converging to  diagrammatic Hochschild cohomology in terms of functor cohomology groups (rather than the more general Baues-Wirshing cohomology groups);

\begin{cor}\label{cor:ssseqwithlims}
Let $\bC$ be a finite category,  $\cA$  a presheaf of $k$-algebras on $\bC$ which is a surjective homological epimorphism. 
    Then, there is a  spectral sequence of cohomological type
    \[
    IE_2^{p,q}=\mathrm{lim}_{\bC}^p \HH^q(\cA(-),\cA(-))
    \]
    converging to the diagrammatic Hochschild cohomology $\HH_{\GS}^*(\cA,\cA)$ of $\cA$.    
\end{cor}

The advantage of Corollary~\ref{cor:ssseqwithlims} is that it allows us to use properties of functor cohomology in computing the pages of the spectral sequence. In fact, we have the following direct corollary;

\begin{cor}\label{cor:spseqterminal}
    Let $\bC$ be a category with terminal object $\bar{c}$, and $\cA$  a presheaf of $k$-algebras on $\bC$ which is a surjective homological epimorphism.
    Then, the spectral sequence $IE_r^{*,*}$ has 2-term given by
    \[
    IE_2^{p,q}=
    \begin{cases}
    0 &\text{ if } p>0\\
    \HH^q(\cA(\bar{c}),\cA(\bar{c})) &\text{ otherwise} 
    \end{cases}
    \]
    for all $p,q \in \N$. 
\end{cor}

\begin{proof}
    The statement follows from Remark~\ref{rem:fucntorterminal} and Corollary~\ref{cor:ssseqwithlims}:  functor cohomology, which agrees with the higher limits by~\cite{Gabriel1967CalculusOF},  is $0$ in positive degrees, and it is given by the evaluation of the functor at the terminal object in degree~0. 
\end{proof}

In particular, if the category $\bC$ has a terminal object, and $\cA$ is a presheaf of $k$-algebras on $\bC$ which is a surjective homological epimorphism, then the associated spectral sequence collapses at the second page, where it agrees with the standard Hochschild cohomology of the terminal object. As a consequence, the target diagrammatic Hochschild cohomology also reduces to the standard Hochschild cohomology of the terminal object. This fact extends the preliminary result of Corollary~\ref{cor:highercolumns} to all (not necessarily free)  categories with a terminal object -- but at the cost of restricting the class of algebra morphisms at hand. 

\begin{rem}\label{rem:colimits}
    Recall that the higher (co)limits compute the (co)homology groups of the homotopy (co)limits, hence Corollary~\ref{cor:spseqterminal} can be stated by saying that $IE_2^{*,q}$ is given by the cohomology groups of the homotopy limit of $\HH^q(\cA(-),\cA(-))$. One may ask whether this is the same as computing the limit, hence whether the second page can be described in terms of Hochschild cohomology of the limit algebra (which is $\cA(\overline{c})$ in the case of a terminal object $\overline{c}$). However, this is not the case because Hochschild cohomology does not commute with taking limits of algebras. For counterexamples,  
    see, e.g.~\cite [Example~1.2.1]{zbMATH05574750}.  
\end{rem}

Let $\phi\colon B\to A$ be a surjective homological epimorphism and  $[1]$  the arrow category  $0\to 1$.

\begin{example}
     Consider the presheaf $\cA\colon [1]^{\op}\to\Alg^{\Hepi}_k $ defined by $\cA(0)=A$, $\cA(1)=B$, and on the unique non-trivial morphism  by $\cA_{0,1}=\cA(1\to 0)=\phi$. To compute the diagrammatic Hochschild cohomology of $\cA$, consider the natural system $\cHH^q(\cA,\cA)$ of Eq.~\eqref{eq:natusystHH}. By Theorem~\ref{thm:selfduality}, this is equivalent to considering the system $\HH^q(\cA,\cA)=\HH^q\circ \cA \circ \pi$, where the projection functor $\pi\colon \Tw\bC\to \bC^\op$ sends a morphism $c\to d$ in $[1]$ to the source~$c$.  The associated spectral sequence $IE_r^{*,*}$  collapses at the second page. The $0$-th column of the first page is given by $IE_1^{0,*}=\HH^*(A)\oplus\HH^*(B)$, whereas the non-degenerate part of the first column is given by $IE_1^{1,*}=\HH^*(B,A)$; note that this group can be identified with $ \HH^*(A,A)$ as $\phi\colon B\to A$ is a surjective homological epimorphism.
    A direct computation shows that the second page of the spectral sequence has only one non-trivial column, which is the $0$-th column; in particular, we have $IE_2^{0,*}=\HH^*(B)$; as also predicted by Corollary~\ref{cor:spseqterminal}. The object~$1$ is indeed the terminal object of the category $[1]$, and the diagrammatic Hochschild cohomology of $\cA$ agrees with the Hochschild cohomology of $B$. 
\end{example}

Arguing as in the previous example, more generally, we get:

\begin{example}
    Let  $\bC=[n]$ be the category from Example~\ref{ex:[n]}, and $\cA$ a surjective homological epimorphism on $[n]$. Then  the diagrammatic Hochschild cohomology $\HH_{\GS}^*(\cA,\cA)$ is the classical Hochschild cohomology of $\cA(n)$.
\end{example}

\section{Diagrams of  incidence algebras}\label{sectionapplications}

Motivated by applications in applied topology, in this section, we use the general theory of homological epimorphisms to the special case of diagrams of incidence algebras (of simplicial complexes). 

Let $P$ be a finite poset. 
Observe that it can be seen as a category $\mathbf{P}$ with the same objects of $P$, by declaring a morphism $p\to q$ in $\mathbf{P}$ if and only if $p\leq q$ in $P$. The nerve of $P$ is defined as the nerve of the category $\mathbf{P}$; that is, the simplicial complex $\Nerve(P)$ in which the $n$-simplices are given by the chains $\sigma=p_0<\dots<p_n$ of elements of $P$. 
Let $k$ be a field. To a category $\bC$ we associate a $k$-algebra $k\bC$ as follows. 

\begin{defn}
   The \emph{category algebra} $k\bC$ is the free $k$-module with basis the set of morphisms of $\bC$. The product on the basis elements is given by
    \[f \cdot g = 
    \begin{cases}
        f \circ g & \text{when the composition exists in $\bC$}\\
        0 & \text{otherwise}    
    \end{cases}\]
    and then it is linearly extended to the whole $k\bC$. 
\end{defn}

Examples of category algebras are the classical group algebras -- viewing a group as a category with a single object -- and the incidence algebras of posets~\cite{rota}. In Rota's original definition, the incidence algebra of a poset $P$ is given by all the possible $k$-valued functions on the intervals of $P$ with product $f\ast g(x,y)\coloneqq \sum_{x\leq z\leq y}f(x,z)g(z,y)$. However, the two definitions are equivalent for (locally) finite posets. In the following, by incidence algebra of a poset~$P$, we shall mean the category algebra of the (category associated with the) poset~$P$, and we write $I(P):=kP$. 

\begin{rem}
    The incidence algebra of a poset $P$ is finite-dimensional (as a $k$-algebra) if and only if $P$ is finite.
\end{rem}

\begin{rem}\label{rem:covfunct}
    Taking category algebras is not generally functorial. However, it is functorial if we restrict to functors between categories which are injective on objects~\cite[Proposition~2.2.3]{XU2007153}. In particular, a map of posets $P\to Q$ that is injective induces a morphism $I(P)\to I(Q)$ between the associated incidence algebras.
\end{rem}

Recall that a subset $Q\neq \emptyset$ of a poset $(P, \leq)$ is called a \emph{lower ideal} if for all $x\in Q$ and $y\leq x$, 
then also $y\in Q$. If $Q$ is a lower ideal of $P$, an induced map $I(P)\to I(Q)$ is obtained by restriction. Taking incidence algebras does not extend to a functor on the whole category of (finite) posets, not even if we restrict to injective maps as in Remark~\ref{rem:covfunct}. To sidestep this issue, we consider functors with the unique lifting of factorizations~\cite[Section~5]{zbMATH06126476}:

\begin{defn}
    A functor $F\colon \bC\to\bD$ has \emph{unique lifting of factorizations} (ULF) if whenever $f$ is a map of $\bC$ and $F(f)=g_1\circ g_2$ in $\bD$, then there exist unique maps $f_1$ and $f_2$ in $\bC$ such that $f=f_1\circ f_2$, $F(f_1)=g_1$ and $F(f_2)=g_2$.
\end{defn}

\begin{example}
Let $V$ be a finite set of vertices. As it is easy to verify, there is a correspondence:

\begin{itemize}
  \item Given a simplicial complex $K \subseteq \mathcal{P}(V)$ with $\varnothing \in K$, its inclusion
  \[
    i \colon K \hookrightarrow \mathcal{P}(V)
  \]
  is a ULF functor of posets.
  \item Conversely, given a ULF functor $F \colon P \to \mathcal{P}(V)$ from a poset $P$ with an initial element mapped to $\varnothing$, its image
  \[
    K = \operatorname{Im}(F) \subseteq \mathcal{P}(V)
  \]
  is a simplicial complex (a lower-closed subset).
\end{itemize}

Thus:

\begin{quote}
Simplicial complexes $K \subseteq \mathcal{P}(V)$ are exactly the images of ULF functors $F \colon P \to \mathcal{P}(V)$ from posets with initial element mapping to $\varnothing$.
\end{quote}
\end{example}

In the notation of \cite[Section~5]{zbMATH06126476}, there is a category $\mathbf{Cat}^*$ of finite categories and ULF-functors. Let $\mathbf{Poset}^*$ be the full subcategory of $\mathbf{Cat}^*$ consisting of finite posets (seen as categories) and ULF-functors. Then, the induced restriction maps between the associated incidence algebras yield a contravariant functor
\[
{\mathbf{Poset}^*}^\op \to \mathbf{Alg}_k
\]
sending a poset $P$ to its incidence algebra, and sending an ULF-functor (seen as a map of posets) $F\colon Q\to P$ to the map of incidence algebras $F^*\colon  I(P)\to I(Q)$ induced by restriction. 
More explicitly, $F^*(g)(x,y)=g(F(x),F(y))$ for all $g\in I(P)$ and $x,y\in Q$, with $x\leq y$.
\begin{example}
     Let $P$ be a poset and $Q\subseteq P$ a lower ideal. Then, the inclusion map $\iota\colon Q\to P$ is ULF.
\end{example}

The following result was previously shown in  \cite[Section~5.7]{suárezalvarez2007applications}:

\begin{prop}\label{prop:homepiincalg}
    Let $P$ be a poset and $Q\subseteq P$ a lower ideal. Then, the induced map
    \[
    I(P)\to I(Q)
    \]
between incidence algebras is a homological epimorphism.
\end{prop}

\begin{proof}
    The kernel of the map $I(P)\to I(Q)$ is generated by all paths in $P$ which go through a vertex in $P\setminus Q$. As $Q\subseteq P$ is a lower ideal, this implies that it is in fact generated by all paths starting at an object of $Q$. If $e$ denotes the idempotent $e\coloneqq \sum_{x\in P\setminus Q} x$, then the kernel of the map $I(P)\to I(Q)$ can be identified with the ideal of $I(P)$ generated by $e$. In particular, it is idempotent and projective as an $I(P)$-module. The result now follows from Lemma~\ref{lemma:homideals}.
\end{proof}

Recall that, given a finite simplicial complex $\Sigma$, the face poset $F(\Sigma)$ of $\Sigma$ is the poset on the set of simplices of $\Sigma$ with partial relation induced by containment: $\sigma \leq \tau$ if and only if $\tau$ is a face of $\sigma$ in $\Sigma$. A simplicial map $\alpha \colon \Sigma\to\Sigma'$ induces an order-preserving map $F(\alpha) \colon F(\Sigma)\to F(\Sigma')$. Furthermore, the image of $F(\alpha)$ is a lower ideal of  $F(\Sigma')$ because they are face posets of simplicial complexes. 
We now consider incidence algebras arising from the face posets of simplicial complexes. For simplicity, we denote by $I(\Sigma)$ the incidence algebra of the face poset of~$\Sigma$. If $\Sigma\subseteq \Sigma'$ is a subcomplex, then we get a map $I(\Sigma')\to I(\Sigma)$ and a  functor
\begin{equation}\label{eq:functorI}
I\colon \mathbf{SimpComp}^{\op}\to \mathbf{Alg}_k  
\end{equation}
from the (opposite) category of finite simplicial complexes and injective simplicial maps, to $k$-algebras. By Proposition~\ref{prop:homepiincalg}, the functor $I$ refines to a functor  $\mathbf{SimpComp}^{\op}\to \Alg_k^\Hepi$.

Hochschild cohomology of incidence algebras of simplicial complexes and simplicial cohomology are closely related~\cite{GERSTENHABER1983143,zbMATH04118524, zbMATH06676099}; we recall the precise statement:

\begin{thm}\label{thm:HHis SC}
    There is a functor $\Sigma\mapsto I(\Sigma)$ from finite simplicial complexes to associative unital $k$-algebras inducing an isomorphism
    \[
    \H^*(\Sigma;k)\cong \HH^*(I(\Sigma),I(\Sigma))
    \]
    between the simplicial cohomology of $\Sigma$ and Hochschild cohomology of $I(\Sigma)$. 
\end{thm}

In the same notation of  Example~\ref{ex:[n]}, let $[n]$ be the category with objects the numbers $0,\dots, n$ and morphisms $i\to j$ if and only if $i\leq j$. As customary in topological data analysis, we are interested in filtrations of simplicial complexes;

\begin{defn}
    An $[n]$-indexed filtration in a category $\bD$ is a functor $\mathcal{F}\colon [n]\to \bD $.
\end{defn}

Recall that $\mathbf{SimpComp}$ denotes the category of finite simplicial complexes and injective simplicial maps.
If ${\Sigma}\colon [n]\to \mathbf{SimpComp}$, $i\mapsto \Sigma_i$, is a filtration in simplicial complexes,  we can apply the functor $I$ of Eq.~\eqref{eq:functorI}, getting by composition a diagram $\mathcal{I}$ of $k$-algebras; for each $i\in [n]$, we have an incidence $k$-algebra $\mathcal{I}(i)=I(\Sigma_i)$, and for each $i\leq j$ a functorial restriction map $I(\Sigma_j)\to I(\Sigma_i)$. For simplicity, we refer to such a functor as a \emph{diagram in incidence algebras}.

\begin{thm}\label{thm:towardshighPH}
    Let $\mathcal{I}\colon [n]^{\op}\to \mathbf{Alg}_k$ be a diagram in incidence algebras of finite simplicial complexes with $\mathcal{I}(i)\coloneqq I(\Sigma_i)$. Then,  we have 
    \[
    \HH^q_{\GS}(\mathcal{I}, \mathcal{I})\cong \H^q(\Sigma_n;k)
    \]
    for all $q\geq 0$. 
\end{thm}

\begin{proof}
    The diagram $\mathcal{I}$ is a diagram of incidence algebras of finite simplicial complexes and injective simplicial maps. Therefore, each inclusion $\Sigma_i\to \Sigma_j$ induces an epimorphism between the associated incidence algebras. Furthermore, the induced maps $\mathcal{I}(j)\to \mathcal{I}(i) $ are homological epimorphisms by Proposition~\ref{prop:homepiincalg}. Hence, we can apply Corollary~\ref{cor:ssseqwithlims}. Since the category~$[n]$ has a terminal object~$n$, the second page of the  associated spectral sequence $IE_2^{p,q}$ is given by 
    \[
    IE_2^{p,q}=
    \begin{cases}
    0 &\text{ if } p>0\\
    \HH^q(I(\Sigma_n),I(\Sigma_n)) &\text{ otherwise} 
    \end{cases}
    \]
    and converges to the diagrammatic Hochschild cohomology of $\mathcal{I}$. The obtained  Hochschild cohomology groups are identified with the standard simplicial cohomology groups of $\Sigma_n$ by Theorem~\ref{thm:HHis SC}, and the statement follows.
\end{proof}

Arguing as above, we get the following, more general, statement:

\begin{thm}\label{thm:terminal}
    Let $\bC$ be a finite category with terminal object $\bar{c}$, and let  $\mathcal{I}=I\circ \Sigma\colon \bC^\op \to \Alg^{\Hepi}_k$ be a diagram of incidence algebras of finite simplicial complexes (and injective simplicial maps). Then,  
        \[
    \HH^q_{\GS}(\mathcal{I}, \mathcal{I})\cong \H^q(\Sigma(\bar{c});k)
    \]
    for all $q\geq 0$.
\end{thm}

\subsection{Incidence algebras and colimits.}\label{sec:colim}

In this section, we compare the colimit of a diagram of simplicial complexes with the limit of the associated incidence algebras.

Let $P$ be a finite poset and $\Sigma\colon P\to \mathbf{SimpComp}$ be a diagram of finite simplicial complexes with $\Sigma_p\coloneqq\Sigma(p)$ and  $\Sigma_{pq}\coloneqq\Sigma(p\leq q)\colon \Sigma_p\hookrightarrow \Sigma_q$ injective maps. 
It is easy to see that the colimit \(\operatorname{colim}_P \Sigma\) exists  and is again a finite simplicial complex. In fact, the category of simplicial complexes is cocomplete, and because \(P\) is finite, one may linearly order its elements and build the colimit inductively as a sequence of pushouts:
each step glues one finite complex along a finite subcomplex (the image of its predecessors under the injective maps).
Pushouts of finite simplicial complexes along finite subcomplexes remain finite,
so by induction, the resulting colimit is finite. For completeness, we also provide a more constructive proof.

First, consider the disjoint union of all complexes in the diagram:
\[
\bigsqcup_{p \in P} \Sigma_p \ .
\]
Define an equivalence relation \(\sim\) generated by the identifications
\[
x \sim \Sigma(p \le q)(x) \qquad \text{for all } (p \le q) \in P,\; x \in \Sigma_p \ .
\]
Since every \(\Sigma(p \le q)\) is injective, this relation never identifies distinct points inside the same \(\Sigma_p\), but only glues points of different complexes along the diagram morphisms.
Define the set of vertices of the candidate colimit \(K\) as
\[
V := \left(\bigsqcup_{p \in P} V(\Sigma_p)\right)\big/ \sim \ ,
\]
where $V(\Sigma_p)$ denotes the set of vertices of $\Sigma_p$.
If \(\sigma = \{v_0, \dots, v_k\}\) is a \(k\)-simplex in some \(\Sigma_p\), we set \([\sigma] := \{[v_0], \dots, [v_k]\} \subseteq V\) and define
\[
S(K) := \{\, [\sigma] \mid \sigma \text{ is a simplex of some } \Sigma_p \,\} \ .
\]
This is well-defined because the diagram maps are simplicial: if \(\sigma\) is a simplex in \(\Sigma_p\) and \(p \le q\), then \(\Sigma(p \le q)(\sigma)\) is a simplex in \(\Sigma_q\), and these are identified under \(\sim\).
Moreover, \(S(K)\) is closed under taking faces: if \([\sigma] \in S(K)\) and \(\tau\) is a face of \(\sigma\) in some \(\Sigma_p\), then \([\tau] \in S(K)\).
Hence \(K = (V, S(K))\) is a simplicial complex.

For each \(p \in P\) there is a canonical inclusion
\[
\iota_p \colon\Sigma_p \longrightarrow K
\]
sending a vertex \(v \in \Sigma_p\) to its equivalence class \([v]\) and it is extended simplicially.
These maps satisfy the relation \(\iota_q \circ \Sigma(p \le q) = \iota_p\).
Now let \(L\) be a simplicial complex and \(\phi_p \colon \Sigma_p \to L\) a compatible cone such that
\(\phi_q \circ \Sigma(p \le q) = \phi_p\).
Define \(\Phi \colon K \to L\) on vertices by
\[
\Phi([v]) := \phi_p(v) \qquad (v \in V(\Sigma_p)) \ .
\]
This is well defined because if \(v \sim \Sigma(p \le q)(v)\), then \(\phi_q(\Sigma(p \le q)(v)) = \phi_p(v)\).
Extend \(\Phi\) simplicially.
Then \(\Phi \circ \iota_p = \phi_p\), and \(\Phi\) is unique with this property. That is,  \(K\) satisfies the universal property of the colimit of \(\Sigma\). Furthermore, observe that:

\begin{itemize}
    \item The vertex set \(V\) of \(K\) is a quotient of the finite disjoint union \(\bigsqcup_{p\in P} V(\Sigma_p)\).  
    Since each \(V(\Sigma_p)\) is finite and \(P\) is finite, this union is finite.  
    Therefore \(V\) is finite.
    \item Similarly, the set of simplices \(S(K)\) is a quotient of the finite disjoint union \(\bigsqcup_{p\in P} S(\Sigma_p)\), hence finite.
\end{itemize}

Therefore \(K\) has finitely many vertices and finitely many simplices, so it is a finite simplicial complex.

\medskip
We now come back to the relation between the colimit of a diagram of simplicial complexes and the limit of the associated incidence algebras. Let \(P\) be a finite poset and 
$\Sigma\colon P\to \mathbf{SimpComp}$ is a functor to the category of finite simplicial complexes
and injective simplicial maps. Let $K$ be its colimit.
For each \(p\in P\) we have an injective simplicial map
\[
i_p \colon \Sigma_p \hookrightarrow K,
\]
which induces an injective order-embedding
\[
\iota_p \colon {F}(\Sigma_p) \hookrightarrow {F}(K)
\]
between the corresponding face posets.
Let \(I(X)\) denote the incidence algebra (over \(k\)) of a finite poset~\(X\).
For each relation \(p\le q\) in \(P\), the inclusion
\({F}(\Sigma_p)\hookrightarrow {F}(\Sigma_q)\)
induces a restriction homomorphism
\[
r_{q\to p} \colon I({F}(\Sigma_q)) \longrightarrow I({F}(\Sigma_p))\ ,
\]
so that we obtain a contravariant diagram
\[
\Sigma^* \colon P^{\op} \longrightarrow \mathbf{Alg}_k\ .
\]
Restriction along \(\iota_p\) gives a map 
\[
\rho_p \colon I({F}(K)) \longrightarrow I({F}(\Sigma_p)) \ .
\]

\begin{lemma}
If \(\sigma \le \tau\) in \({F}(K)\), then there exists \(p\in P\) such that
\(\tau \in \iota_p({F}(\Sigma_p))\),
and for this \(p\) the entire interval
\([\sigma,\tau] = \{x : \sigma\le x\le \tau\}\)
is contained in \(\iota_p(F(\Sigma_p))\).
\end{lemma}

\begin{proof}
    Every simplex \(\tau\) in the colimit \(K\) comes from some \(\Sigma_p\)
since the diagram maps are injective, and we form a colimit involving injective simplicial maps.
All faces of \(\tau\) lie in the same \(\Sigma_p\),
hence so does the whole interval \([\sigma,\tau]\).
\end{proof}

This lemma shows that every interval in \(F(K)\) is already
an interval inside some \(F(\Sigma_p)\).
Consequently, the family \((\rho_p)_p\) defines a homomorphism
\[
\Theta \colon I(F(K)) \longrightarrow \varprojlim\nolimits_{p\in P} I(F(\Sigma_p)) \ .
\]
\begin{prop}\label{prop:incid_lim}  
For a diagram $\Sigma\colon P\to \mathbf{SimpComp}$ of finite simplicial complexes with injective simplicial maps,
the inverse limit $\varprojlim\nolimits_{p\in P} I(F(\Sigma_p))$ of the corresponding incidence algebras (with restriction maps)
is canonically isomorphic, as a \(k\)-algebra, to the incidence algebra $I(F(K)) $ of the colimit complex.
\end{prop}

\begin{proof}
    We show that $\Theta \colon I(F(K)) \longrightarrow \varprojlim\nolimits_{p\in P} I(F(\Sigma_p))$ is an isomorphism of $k$-algebras. We have already shown that it is a homomorphism. It is to be shown that it is injective and surjective.

\textbf{Injectivity.}
If \(f\neq g\in I(F(K))\),
pick an interval \([\sigma,\tau]\) with \(f(\sigma,\tau)\neq g(\sigma,\tau)\).
Choose \(p\) with \(\tau\in \iota_p(F(\Sigma_p))\);
then \([\sigma,\tau]\subset \iota_p(F(\Sigma_p))\),
so \(\rho_p(f)\neq\rho_p(g)\).
Thus \(\Theta(f)\neq \Theta(g)\).

\medskip
\textbf{Surjectivity.}
Given a compatible family \((f_p)_p \in \varprojlim I(F(\Sigma_p))\),
define \(f\in I(\mathcal{F}(K))\) by
\[
f(\sigma,\tau) := f_p(\sigma,\tau)
\quad\text{for any }p\text{ such that }\tau\in \iota_p(F(\Sigma_p)).
\]

This is well-defined because if $(x_p,y_p)$ and $(x_q,y_q)$ both represent the same interval $([\sigma],[\tau])$ in $F(K)$, they are connected by a finite zig-zag of images and preimages along the simplicial maps in the diagram $\Sigma$.
Along each step $r\leq s$ of that zig-zag, the compatibility condition in the limit $\Sigma^*_{r,s}(f_s) = f_r$ is what we need to transport equalities of values along the chain.

The convolution product is respected:
for \(\sigma\le\tau\), all chains
\(\sigma=x_0\le\cdots\le x_m=\tau\)
lie inside the same \(F(\Sigma_p)\),
so \((f*g)(\sigma,\tau)\) is computed in \(I(F(\Sigma_p))\)
and agrees with the product there.

Therefore \(\Theta\) is an isomorphism of \(k\)-algebras. 
\end{proof}

\begin{rem}
    The argument relies crucially on the fact that comparability
\(\sigma\le\tau\) in a face poset forces both elements,
and indeed the entire interval, to lie in a single simplicial subcomplex.
This ``interval locality'' may fail for arbitrary posets glued along subposets,
so the statement is special to face posets of simplicial complexes.

We observe here also that whether or not the empty simplex is included in each face poset
does not affect the result, provided the convention is consistent
across the diagram.
\end{rem}

Let $P$ be a finite poset without a terminal object, and $\Sigma$ and $\Sigma^*$ be as above. By using Theorem \ref{thm:terminal}, we can compute the diagrammatic cohomology of the diagram $\hat{\Sigma}^*$ obtained by adding a terminal object to $P$ and decorating it with the limit of $\Sigma^*$, with the obvious maps. One may wonder if this cohomology (which is nothing but the Hochschild cohomology of the limit) is isomorphic or not to the diagrammatic cohomology of the original diagram. In general, they are not isomorphic, as we show in the following example.

\begin{example}
    Consider the poset $P$:
\begin{equation}
 \begin{tikzcd}
\{1\} \arrow[r] \arrow[d] 
& \{4\}  \\
 \{3\}  & \{2\} \arrow[l] \arrow[u]
 \end{tikzcd}    
 \end{equation} 
We set $\Sigma\colon P\to \mathbf{SimpComp}$ given by $\Sigma(p)=\{\ast\}$, the singleton, for all $p$, and $\Sigma_{pq}=\mathrm{id}$, for all $p\leq q$:
\begin{equation}
 \begin{tikzcd}
\{\ast\} \arrow[r] \arrow[d] 
& \{\ast\}  \\
 \{\ast\}  & \{\ast\} \arrow[l] \arrow[u]
 \end{tikzcd}    
 \end{equation} 
Then, using \cite{GERSTENHABER1983143}, it follows that $\HH^n_{\GS}=k$ for $n=0,1$, and $\HH^n_{\GS}={0}$ for $n\geq 2$.

Extend now $P$ to $\hat{P}$ by adding a maximal element:
\begin{equation}    
\begin{tikzcd}
\{1\} \arrow[rr] \arrow[dd] \arrow[dr] 
  &  & \{4\} \arrow[dl] \\
  & \{5\} & \\
\{3\} \arrow[ur] 
  &  & \{2\} \arrow[uu] \arrow[ll] \arrow[ul]
\end{tikzcd}
\end{equation}
We extend $\Sigma $ to $\hat{\Sigma}\colon \hat{P}\to \mathbf{SimpComp}$ in the obvious way, namely

\begin{equation}    
\begin{tikzcd}
\{\ast\} \arrow[rr] \arrow[dd] \arrow[dr] 
  &  & \{\ast\} \arrow[dl] \\
  & \{\ast\} & \\
\{\ast\} \arrow[ur] 
  &  & \{\ast\} \arrow[uu] \arrow[ll] \arrow[ul]
\end{tikzcd}
\end{equation}

Then, again, by using \cite{GERSTENHABER1983143}, it follows that $\HH^n_{\GS}=k$ for $n=0$, and $\HH^n_{\GS}={0}$ for $n\geq 1$.

\end{example}

This example shows that adding a maximal element, or terminal object, may change the diagrammatic Hochschild cohomology groups substantially even in easy cases -- see also Remark~\ref{rem:colimits}. 

\subsection{Further perspectives}
We now come to our motivations and applications to topological data analysis. First, we recall that incidence algebras are objects of great interest in persistent homology and topological data analysis; see, for example, the recent works~\cite{{zbMATH07701447}, zbMATH07787684,zbMATH07844814}. Our motivating question in dealing with diagrams of incidence algebras concerns the existence of higher persistent modules associated with a filtration of simplicial complexes. The bridge between the theory developed in the first sections and persistence relies on the following observation:

\begin{rem}
    The first column of the first page of the spectral sequence $IE_1^{*,*}$ yields the (classical) persistent module associated to the filtration $\Sigma\colon [n]\to \mathbf{SimplCompl}$.
\end{rem} 
Then, Theorem~\ref{thm:towardshighPH} implies that, when turning the page of the spectral sequence in the second page, the only surviving groups are the cohomology groups of the terminal object of the filtration. This is not the case for arbitrary indexing diagrams, where the diagram's topology influences the diagrammatic Hochschild cohomology. In complete analogy with what happens with the first column, one may be tempted to call the other columns the  ``higher persistent cohomology groups'' of the diagram. This perspective leads to the   following question: 

\begin{q}
    What do the higher persistent homology groups describe? Can we describe with a closed formula the cohomology groups appearing on the first and second pages of the spectral sequence?
\end{q}

\subsection*{Acknowledgments}
 
LC~was supported by the Starting Grant 101077154 ``Definable Algebraic Topology'' from the European Research Council of Martino Lupini. LC also acknowledges partial support from INdAM-GNSAGA.
This study was carried out within the FAIR - Future Artificial Intelligence Research and received funding from the European Union Next-GenerationEU (PIANO NAZIONALE DI RIPRESA E RESILIENZA (PNRR) – MISSIONE 4 COMPONENTE 2, INVESTIMENTO 1.3 – D.D. 1555 11/10/2022, PE00000013). 
Work supported by the European Union’s Horizon Europe research and innovation program for the project FINDHR (g.a. 101070212). 
This manuscript reflects only the authors’ views and opinions; neither the European Union nor the European Commission can be considered responsible for them.
\bibliographystyle{alpha}
\bibliography{biblio}

\end{document}